\documentclass[11pt,a4paper,leqno]{article}
\usepackage[utf8]{inputenc}
\usepackage[T1]{fontenc}
\usepackage[left=2.5 cm,right=2.5cm,top=2.5cm,bottom=2.5cm]{geometry}
\usepackage{amsthm,amssymb,amsmath,amsfonts,empheq}
\usepackage{mathtools}
\usepackage{color}
\usepackage{xcolor}
\usepackage{array}
\usepackage[shortlabels]{enumitem}
\usepackage{graphicx}
\usepackage{fancyhdr}
\usepackage[french,english]{babel}% gestion des langues
\usepackage{hyperref}
\usepackage{eucal}
\usepackage{capt-of}
\usepackage{moreverb}

\newtheorem{theorem}{\textbf{Theorem}}[section]
\newtheorem{remark}[theorem]{\textbf{Remark}}

\newtheorem{lemma}[theorem]{\textbf{Lemma}}
\newtheorem{corollary}[theorem]{\textbf{Corollary}}
\newtheorem{proposition}[theorem]{\textbf{Proposition}}

\newtheorem{definition}[theorem]{\textbf{Definition}}
\newtheorem*{acknowledgement}{\textbf{Acknowledgement}}

\numberwithin{equation}{section}

\usepackage{titlesec}
\titleformat\section{}{}{0pt}{\Large\scshape\filcenter\thesection{} - }

\def\a{\alpha}
\def\aa{\alpha}
\def\bb{\beta}
\def\dd{\delta}

\def\ll{\lambda}
\def\LL{\Lambda}

\def\th{\theta}
\def\Th{\Theta}

\def\ee{\varepsilon}
\def\vp{\varphi}

\newcommand{\po}{{\overline p}}

\newcommand{\Xo}{{\overline X}}

\newcommand{\qo}{{\overline q}}

\newcommand{\uo}{{\overline u}}
\newcommand{\mo}{{\overline m}}
\newcommand{\muo}{{\overline{\mu}}}
\newcommand{\aao}{{\overline \aa}}

\newcommand{\cB}{{\cal B}}

\newcommand{\RR}{ \mathbb{R}}

\newcommand{\NN}{ \mathbb{N}}

\newcommand{\ZZ}{ \mathbb{Z}}

\newcommand{\EE}{{\mathbb E}}
\newcommand{\TT}{{\mathbb{T}}}

\newcommand{\cL}{{\mathcal L}}

\newcommand{\cP}{{\mathcal P}}

\newcommand{\bt}{{\widetilde b}}

\newcommand{\ut}{{\widetilde u}}

\newcommand{\mt}{{\widetilde m}}

\newcommand{\Xt}{{\widetilde X}}
\newcommand{\pt}{{\widetilde p}}
\newcommand{\qt}{{\widetilde q}}

\newcommand{\aat}{{\widetilde \aa}}

\newcommand{\mut}{{\widetilde \mu}}

\newcommand{\Ct}{{\widetilde C}}

\newcommand{\Kt}{{\widetilde K}}

\newcommand{\psit}{{\widetilde{\psi}}}

\DeclareMathOperator{\argmin}{\mathop{\rm argmin}}

\DeclareMathOperator{\supp}{\mathop{\rm supp}}

\newcommand{\ptt}{{\partial_{t}}}

\newcommand{\norm}[3][]{{\left\| #2 \right\|_{#3}^{#1}}}
\newcommand{\normc}[2]{{\left\| #1 \right\|_{C^{#2}}}}
\newcommand{\norminf}[2][]{{\left\| #2 \right\|_{\infty}^{#1}}}

\DeclareMathOperator{\divo} {div}

\newcommand{\lp}{\left(}
\newcommand{\rp}{\right)}

\newcommand{\lc}{\left\{}
\newcommand{\rc}{\right\}}
\newcommand{\lb}{\left[}
\newcommand{\rb}{\right]}

\newcommand{\labs}{\left|}
\newcommand{\rabs}{\right|}

\title{Mean Field Games with monotonous interactions through the law of states and controls of the agents}
\author{Z. Kobeissi\thanks{Laboratoire Jacques-Louis Lions, Univ. Paris Diderot, Sorbonne Paris Cit\'e, UMR 7598, UPMC, CNRS, 75205, Paris, France. zkobeissi@math.univ-paris-diderot.fr}}

\begin{document}
\maketitle

\begin{abstract}
    We consider a class of Mean Field Games
    in which the agents may interact through
    the statistical distribution of their states and controls.
    It is supposed that the Hamiltonian behaves like a power
    of its arguments as they tend to infinity,
    with an exponent larger than one.
    A monotonicity assumption is also made.
    Existence and uniqueness are proved using
    a priori estimates which stem from the monotonicity assumptions
    and Leray-Schauder theorem. Applications of the results are given.
\end{abstract}

\section{Introduction}
\label{sec:MonoIntro}
The theory of Mean Field Games (MFG for short)
aims at studying deterministic or stochastic differential
games (Nash equilibria) as the number of agents tends to infinity.
It has been introduced
in the independent works of J.M. Lasry and P.L. Lions
\cite{MR2269875,MR2271747,MR2295621},
and of M.Y. Huang, P.E. Caines
and R.Malham{\'e} \cite{MR2352434,MR2346927}.
The agents are supposed to be
rational (given a cost to be minimized, they always choose the optimal strategies), and indistinguishable.
Furthermore, the agents interact via some empirical averages of quantities which depend on the state variable.

The most common Mean Field Game systems,
in which the agents may interact only through
their states
can often be summarized by a system of
two coupled partial differential equations which is named
the MFG system.
On the one hand, 
the optimal value of a generic agent at
some time $t$ and state $x$ is denoted by $u(t, x)$ and is defined as the lowest cost that a representative
agent can achieve from time $t$ to $T$ if it is at state $x$ at time $t$. The value function satisfies a
Hamilton-Jacobi-Bellman equation posed backward in time with a terminal condition involving
a terminal cost.
On the other hand,
there is
a Fokker-Planck-Kolmogorov equation describing the evolution of the statistical distribution
m of the state variable; this equation is a forward in time parabolic equation, and the initial
distribution at time $t = 0$ is given.
Here we take a finite horizon time $T>0$,
and we only consider  second-order nondegenerate MFG systems.
In this case, the MFG system is often written as:
        \begin{equation}
            \label{eq:tradMFG}
            \lc
            \begin{aligned}
            &-\ptt u(t,x)
            - \nu\Delta u(t,x)
            + H(t,x, \nabla_xu(t,x))
            = f(t,x,m(t))
            &\text{ in }
            (0,T)\times\RR^d,\\
            & \ptt m(t,x)
            - \nu\Delta m(t,x)
            -\divo(H_p(t,x,\nabla_xu(t,x))m)
            =0
            &\text{ in }
            (0,T)\times\RR^d,\\
            &u(T,x)
            =g(x,m(T))
            &\text{ in }
            \RR^d,\\
            &m(0,x)
            =m_0(x)
            &\text{ in }
            \RR^d.
        \end{aligned}
            \right.
        \end{equation}
We refer the reader to
\cite{MR3967062}
for some theoretical results on
the convergence of the  $N$-agent Nash equilibrium
 to the solutions of the MFG system.
For a thorough study of the well-posedness
of the MFG system, see the videos
of P.L. Lions' lecture at the
Coll{\`e}ge de France,
and the lecture notes
\cite{Cardaliaguet_notes_on_MFG}.

In this paper we are considering a class of Mean Field Games
in which agents may interact through their states and controls.
To underline this, we choose to use the terminology
{\sl Mean Field Games of Controls (MFGCs)};
this terminology was introduced in \cite{MR3805247}.

Since the agents are assumed to be indistinguishable,
a representative agent may be described
by her state, which is a random process with value in $\RR^d$
denoted by $(X_t)_{t\in[0,T]}$ and satisfying
the following stochastic differential equation,
\begin{equation}
    \label{eq:SDEb}
    dX_t
    =
    b\lp t,X_t,\aa_t,\mu_{\aa}(t)\rp dt
    +\sqrt{2\nu}dW_t,
\end{equation}
where
$X_0$ is a random process
whose law is denoted by $m_0$,
$\lp W_t\rp_{t\in[0,T]}$ is
a Brownian motion on $\RR^d$
independefn with $X_0$,
and $\aa_t$ is the control
chosen by the agent at time $t$.
The diffusion coefficient $\nu$
is assumed to be uncontrolled, constant and positive.
The drift $b$ naturally depends on the control,
and may also depend on the time, the state,
and on the mean field interactions of all agents
through $\mu_{\aa}$ the joint distribution of states and controls.
At the equilibrium $\mu_{\aa}$ should be the law
of the state and the control of the representative agent,
i.e. $\mu_{\aa}(t)=\cL\lp X_t,\aa_t\rp$,
for $t\in[0,T]$.
The aim of an agent is to minimize
the functional given by,
\begin{equation}
    \EE\lb
    \int_0^T L\lp t,X_t,\aa_t,\mu_{\aa}(t)\rp
    +f\lp t, X_t,m(t)\rp dt
    +g\lp X_T,m(T)\rp
    \rb,
\end{equation}
where $m(t)$ is the distribution of agents at time $t$,
which should satisfy $m(t)=\cL\lp X_t\rp$
at the equilibrium.
The coupling function $f$ and the terminal cost $g$
depend on $m$ in a nonlocal manner.
From $L$ the Lagrangian and $b$ the drift,
we define $H$ the Hamiltonian by,
\begin{equation}
    \label{eq:defHMono}
    H\lp t,x,p,\mu_{\aa}\rp
    =
    \sup_{\aa\in\RR^d}
    -p\cdot b\lp t,x,\aa,\mu_{\aa}\rp
    -L\lp t,x,\aa,\mu_{\aa}\rp,
\end{equation}
for $(t,x)\in[0,T]\times\RR^d$,
$p\in\RR^d$
and $\mu_{\aa}\in\cP\lp\RR^d\times\RR^d\rp$,
where $\cP\lp\RR^d\times\RR^d\rp$ is the set
of probability measures on $\RR^d\times\RR^d$.
Under some assumptions on $b$ and $L$ that will be introduced later,
there exists a unique $\aa$ which achieves the supremum
in the latter equality and it also satisfies,
\begin{equation*}
    b\lp t,x,\aa,\mu_{\aa}\rp
    =
    -H_p\lp t,x,p,\mu_{\aa}\rp.
\end{equation*}
In an attempt to keep this paper easy to read,
we introduce $\mu_b$ as the joint law of states and drifts
defined by
\begin{equation}
    \label{eq:mubmuaa}
    \mu_b(t)
    =
    \Bigl[
    \lp x,\aa\rp
    \mapsto
    \lp x,
    b\lp t,x,\aa,\mu_{\aa}(t)\rp\rp
    \Bigr]{\#}\mu_{\aa}(t).
\end{equation}
We believe that the fixed point relation satisfied by $\mu_{\aa}$
at equilibrium is more clear if we distinguish $\mu_{b}$
from $\mu_{\aa}$.
We assume that $b$ is invertible with respect
to $\aa$ in such a way that its inverse map can be expressed
in term of $\mu_b$ instead of $\mu_{\aa}$,
see Assumption \ref{hypo:binvert} below.
This allows
us to define
$\aa^*:[0,T]\times\RR^d\times\RR^d\times\cP\lp\RR^d\times\RR^d\rp
\to\RR^d$ such that
\begin{equation*}
    \bt
    =
    b\lp t,x,\aa^*\lp t,x,\bt,\mu_b\rp,\mu_{\aa}\rp
\end{equation*}
for any $(t,x)\in[0,T]\times\RR^d$,
$\bt\in\RR^d$
and $\mu_{\aa},\mu_b\in\cP\lp\RR^d\times\RR^d\rp$
satisfying \eqref{eq:mubmuaa}.
Conversely,
for any $\aa\in\RR^d$
we have
\begin{equation*}
    \aa
    =
    \aa^*\lp t,x,b\lp t,x,\aa,\mu_{\aa}\rp,\mu_{b}\rp,
\end{equation*}
since $b\lp t,x,\cdot,\mu_{\aa}\rp$ is injective.
This implies that the equality \eqref{eq:mubmuaa}
can be inverted  to express $\mu_{\aa}$ in term of $\mu_b$
and we obtain \eqref{eq:muaa} below.
Within this framework, the usual MFG system \eqref{eq:tradMFG}
is replaced by the following Mean Field Game of Controls (MFGC for short) system,
\begin{subequations}\label{eq:MFGCb}
     \begin{empheq}{align}
            \label{eq:HJBb}
            &-\ptt u(t,x)
            - \nu\Delta u(t,x)
            + H\lp t,x, \nabla_xu(t,x),\mu_{\aa}(t)\rp
            =
            f(t,x,m(t))
            &\text{ in }
            (0,T)\times\RR^d,
            \\
            \label{eq:FPKb}
            & \ptt m(t,x)
            - \nu\Delta m(t,x)
            -\divo\lp H_p\lp t,x,\nabla_xu(t,x),\mu_{\aa}(t)\rp m\rp
            =
            0
            &\text{ in }
            (0,T)\times\RR^d,
            \\
            \label{eq:muaa}
            &\mu_{\aa}(t)
            =
            \Bigl[
            \lp x,\bt\rp
            \mapsto
            \lp x,
            \aa^*\lp t,x,\bt,\mu_{b}(t)\rp\rp
            \Bigr]{\#}\mu_{b}(t)
            &\text{ in } [0,T],
            \\
            \label{eq:mub}
            &\mu_b(t)
            =
            \Bigl(
            I_d,
            -H_p\lp t,\cdot,\nabla_xu(t,\cdot),\mu_{\aa}(t)\rp
            \Bigr){\#}m(t)
            &\text{ in } [0,T],
            \\
            \label{eq:CFub}
            &u(T,x)
            =g(x,m(T))
            &\text{ in }
            \RR^d,\\
            \label{eq:CImb}
            &m(0,x)
            =m_0(x)
            &\text{ in }
            \RR^d.
    \end{empheq}
\end{subequations}
The structural assumption
under which we prove existence and uniqueness
of the solution to \eqref{eq:MFGCb}
is that $L$ satisfies the following
inequality,
\begin{equation*}
    \int_{\RR^d\times\RR^d}
    \lp L\lp t, x,\a,\mu^1 \rp
    -L\lp t, x,\a,\mu^2\rp \rp
    d\lp \mu^1-\mu^2 \rp (x,\aa)
    \geq
    0.
\end{equation*}
for any $t\in[0,T]$,
$\mu^1, \mu^2 \in
\cP\lp\RR^d\times \RR^d \rp$.
This is the Lasry-Lions monotonicity assumption
extended to MFGC that will be referred to as
\ref{hypo:LMono}.
This assumption
is particularly adapted to
applications in economics or finance.

This work follows naturally the analysis
in \cite{2019arXiv190411292K}
in which a MFGC system in the
$d$-dimensional torus and with $b=\aa$
is considered.
In \cite{2019arXiv190411292K},
the monotonicity assumption is replaced by another structural
assumption,
namely that the optimal control $-H_p$
is a contraction with respect to the second marginal
of $\mu$ (when the other arguments and the first marginal 
are fixed)
and that it is bounded
by a quantity that 
depends linearly on the second marginal
of $\mu$ with a coefficient
smaller than $1$.

\subsection*{Related works}
Monotonicity assumptions for MFGC like \ref{hypo:LMono}
have already been discussed 
in \cite{MR3805247,MR3752669,MR3112690}.
In \cite{MR3112690},
the authors proved uniqueness of the solution
to \eqref{eq:MFGCb} with
$b=\aa$ and $\nu=0$ when it exists.
In \cite{MR3752669}
Section $4.6$,
existence and uniqueness
%of solutions to \eqref{eq:MFGCb}
are proved
in the quadratic case with a
uniformly convex Lagrangian and
under an additional linear growth assumption on $H_x$.
In \cite{MR3805247},
the existence of weak solutions
to a MFGC system with
a possibly degenerate diffusion operator
is proved assuming that the inequalities satisfied by
$H$, its derivatives or the optimal control
(here defined in \ref{hypo:binvert} as $\aa^*$),
are uniform with respect to
the joint law of states and controls $\mu_{\aa}$.

A particular application of MFGC satisfying
\ref{hypo:LMono},
namely the Bertrand and Cournot competition
for exhaustible ressource
described in paragraph \ref{subsec:Monoexhau},
has been broadly investigated in the literature.
Let us mention a non exhaustive list of such works:
\cite{2019arXiv190205461F,MR3359708,MR3755719,
MR2762362,MR4064472}.
%This application is often referred as
%the Bertrand and Cournot competition games,
%in reference to the pioneer papers
%of Bertrand \cite{Bertrand}
%and Cournot \cite{cournot1838recherches}.
Its mean field version has been introduced in
\cite{MR2762362},
and obtained from the $N$-agent game in
\cite{MR3359708}
in the case of a linear supply-demand function.
A generalization to the multi-dimensional case
is discussed in \cite{2019arXiv190205461F},
and an extension to negatively correlated
ressources is addressed in \cite{2019arXiv190411292K}.
%In section \ref{sec:appli} we propose
%an extension of their results to more
%general Lagrangian and supply-demand function.

A class of MFGC in which the Lagrangian
depends separately
on $\aa$ and $\mu$,
has been investigated in
\cite{MR3325272}
and \cite{MR3752669}.
In this case,
\ref{hypo:LMono} is naturally
satisfied since the left-hand side
of the inequality is identically equal to $0$.
An existence result is proved in \cite{MR3325272}
under the additional assumption that the set of admissible
controls is compact.
The existence of solution is also proved
in \cite{MR3752669}
Theorem $4.65$ when
the dependency of $L$ upon $\mu$
is uniformly bounded with respect to $\mu$.

The non-monotone case has been studied
in \cite{MR3160525,2019arXiv190411292K}.
In \cite{MR3160525},
an existence result is proved in the
stationnary setting and under
the assumption that the dependence
of $H$ on $\mu$ is small.
In \cite{2019arXiv190411292K},
the existence of solutions to the MFGC system
in the $d$-dimensional torus and with $b=\aa$
is discussed under similar growth assumptions as here.
By and large, existence
of solutions to a MFGC system posed on the $d$-dimensional torus
and with $b=\aa$
was proved in \cite{2019arXiv190411292K} 
in any of the following cases:
\begin{itemize}
    \item short time horizon,
    \item small enough parameters,
    \item weak dependency of $H$ upon $\mu$,
    \item weak dependency of $H_x$ upon $\mu$,
\end{itemize}
and uniqueness is proved only for a short time horizon.
Indeed without
a monotonicity assumption, it is unlikely
that uniqueness holds in general,
numerical examples of non-uniqueness of solutions
to discrete approximations of \eqref{eq:MFGCb}
with $b=\aa$ and in a bounded domain
are showed in \cite{achdou2020mean}.

\subsection*{Organization of the paper}
In Section \ref{sec:MonoAssumptions},
the notations and the assumptions are described,
the case when the control is equal to the drift is discussed.
The main results of the paper,
namely the existence and uniqueness of solution
to \eqref{eq:MFGCb}, are stated in paragraph \ref{subsec:mainMono}.
We give some insights on our strategy for proving
the main results in paragraph \ref{subsec:outlineMono}.
Two applications of the MFGC system \eqref{eq:MFGCb}
are presented in Section \ref{sec:appli}.
Section \ref{sec:MonoFPV} is devoted to solving the
fixed point relation in the joint law of state and control
in the particular case when the drift is equal to the control.
Section \ref{sec:aprioriMono} consists in
giving a priori estimates
for a MFGC system posed 
on the $d$-dimensional torus.
In Section \ref{sec:MainMono},
we prove existence and uniqueness of the solution
to \eqref{eq:MFGCb}
and of an intermediate MFGC system.
%to the MFGC system stated
%on the torus and introduced in Section \ref{sec:aprioriMono},
%then we prove the well-posedness of
%\eqref{eq:MFGCb} when the control is equal to the drift,
%and finally in paragraph \ref{subsec:Exib} we prove
%the main results.
%first we prove existence of solutions to \eqref{eq:MFGCMono}
%using the Leray-Schauder fixed point theorem
%and the a priori estimates of Section \ref{sec:aprioriMono};
%then we extend the existence result to the case where
%the state set is $\RR^d$ instead of $\TT^d_a$;
%then we prove a uniqueness result for
%\eqref{eq:MFGCrd} and \eqref{eq:MFGCMono};
%finally we extend the existence and uniqueness results to \eqref{eq:MFGCb}.

\section{Assumptions}
\label{sec:MonoAssumptions}

\subsection{Notations}
\label{subsec:notaMono}
%Let us define $q\in(1,\infty)$
%an exponent,
%and $q'=\frac{q}{q-1}$ its conjugate exponent.
The spaces of probability measures 
are equipped with the weak* topology.
We denote by
$\cP_2\lp\RR^d\rp$
the subset of
$\cP\lp\RR^d\rp$
of probability measures
with finite second moments,
and
$\cP_{\infty}\lp\RR^d\times\RR^d\rp$
the subset of measures $\mu$ in $\cP\lp\RR^d\times\RR^d\rp$
%with a first marginal in $\cP_2\lp\RR^d\rp$
with a second marginal compactly supported.
%i.e.
%$\mu\lp\cdot\times\RR^d\rp\in\cP_2\lp\RR^d\rp$
%and $\supp\lp\mu\lp\RR^d\times\cdot\rp\rp$ is bounded in $\RR^d$.
For $\mu\in\cP_{\infty}\lp\RR^d\times\RR^d\rp$
and $\qt\in[1,\infty)$,
we define the quantities
$\LL_{\qt}(\mu)$ and $\LL_{\infty}(\mu)$ by,
\begin{equation}
    \label{eq:defL}
\begin{aligned}
    \LL_{\qt}(\mu)
    &=
    \lp\int_{\RR^d\times\RR^d}
    \labs\aa\rabs^{\qt}d\mu\lp x,\alpha\rp\rp^{\frac1{\qt}},
    \\
    \LL_{\infty}(\mu)
    &=
    \sup\lc\labs\aa\rabs,
    (x,\aa)\in\supp\mu\rc.
\end{aligned}
\end{equation}
%$\cP\lp\TT^d_a\times\RR^d\rp$,
For $R>0$, we denote by
$\cP_{\infty,R}\lp\RR^d\times\RR^d\rp$
the subset of measures $\mu$ in
$\cP_{\infty}\lp\RR^d\times\RR^d\rp$
such that $\LL_{\infty}\lp\mu\rp\leq R$.
The probability measures $\mu_{\aa}$ and $\mu_b$ 
involved in
\eqref{eq:MFGCb},
have a particular form,
since they are the images of a measure
$m$ on $\RR^d$
by $\lp I_d,\aa\rp$ and $\lp I_d,b\rp$ respectively,
where $\aa$ and $b$ are bounded measurable
functions from $\RR^d$ to $\RR^d$;
in particular they are supported on the graph of $\aa$
and $b$ respectively.
For $m\in\cP\lp\RR^d\rp$,
we call
$\cP_m\lp\RR^d\times\RR^d\rp$
the set of such measures.
For $\mu\in\cP_m\lp\RR^d\times\RR^d\rp$,
we set $\aa^{\mu}$ to be the unique 
element of $L^{\infty}\lp m\rp$
such that $\mu=\lp I_d,\aa^{\mu}\rp\#m$.
Here, $\LL_{\qt}(\mu)$ and $\LL_{\infty}(\mu)$
defined in \eqref{eq:defL} are given by
\begin{equation}
    \label{eq:defLbis}
\begin{aligned}
    \LL_{q'}(\mu)
    &=
    \norm{\aa^{\mu}}{L^{q'}(m)},
    \\
    \LL_{\infty}(\mu)
    &=
    \norm{\aa^{\mu}}{L^{\infty}(m)}.
\end{aligned}
\end{equation}
Let $C^{0}\lp[0,T]\times\RR^d;\RR^n\rp$ be the set
of bounded continuous functions from $[0,T]\times\RR^d$
to $\RR^n$, for $n$ a positive integer.
We define $C^{0,1}\lp[0,T]\times\RR^d;\RR\rp$ as the set
of the functions
$v\in C^0\lp[0,T]\times\RR^d;\RR\rp$ differentiable at any point with respect to
the state variable, and whose its gradient $\nabla_xv$ is in
$C^{0}\lp[0,T]\times\RR^d;\RR^d\rp$
the set of continuous functions
from $[0,T]\times\RR^d$ to $\RR^d$.
We shall have the use of the parabolic spaces
of Hölder continuous functions
$C^{\frac{\bb}2,\bb}([0,T]\times\RR^d;\RR^n)$
defined for any $\bb\in(0,1)$ and $n\geq 1$ by,
\begin{equation*}
    C^{\frac{\bb}2,\bb}\lp[0,T]\times\RR^d;\RR^n\rp 
    = \lc
    \begin{aligned}
        v& \in C^0([0,T]\times\RR^d;\RR^n),
        \exists C>0 \text{ s.t. }
        \forall (t_1,x_1),(t_2,x_2)\in [0,T]\times\RR^d, \\
        &|v(t_1,x_1)-v(t_2,x_2)|\leq
        C\lp|x_1-x_2|^2 +|t_1-t_2|\rp^{\frac{\bb}2} 
    \end{aligned}
    \rc.
\end{equation*}
This is a Banach space equipped with the norm,
\begin{equation*}
    \norm{v}{C^{\frac{\bb}2,\bb}} 
    =
    \norminf{v}
    +\sup_{(t_1,x_1)\neq(t_2,x_2)}
    \frac{|v(t_1,x_1)-v(t_2,x_2)|}
    {\lp|x_1-x_2|^2 +|t_1-t_2|\rp^{\frac{\bb}2}}.
\end{equation*}
Then we introduce the Banach space
$C^{\frac{1+\bb}2,1+\bb}([0,T]\times\RR^d;\RR)$
for $\bb\in(0,1)$
as the set of the functions
$v\in C^{0,1}([0,T]\times\RR^d;\RR)$
such that $\nabla_xv\in C^{\frac{\bb}2,\bb}\lp[0,T]\times\RR^d;\RR^n\rp$
and which admits a finite norm defined by,
\begin{equation*}
    \normc{v}{\frac{1+\bb}2,1+\bb}
    =
    \norminf{v}
    +\norm{\nabla_xv}{C^{\frac{\bb}2,\bb}}
    + \sup_{(t_1,x)\neq (t_2,x)\in [0,T]\times\RR^d}
    \frac{|v(t_1,x)-v(t_2,x)|}{|t_1-t_2|^{\frac{1+\bb}2}}.
\end{equation*}

When the drift $b$ is equal to the control $\aa$,
\eqref{eq:MFGCb}
can be simplified in the following system,
\begin{subequations}\label{eq:MFGCrd}
     \begin{empheq}{align}
            \label{eq:HJBrd}
            &-\ptt u(t,x)
            - \nu\Delta u(t,x)
            + H\lp t,x, \nabla_xu(t,x),\mu(t)\rp
            =
            f(t,x,m(t))
            &\text{ in }
            (0,T)\times\RR^d,\\
            \label{eq:FPKrd}
            & \ptt m(t,x)
            - \nu\Delta m(t,x)
            -\divo\lp H_p\lp t,x,\nabla_xu(t,x),\mu(t)\rp m\rp
            =
            0
            &\text{ in }
            (0,T)\times\RR^d,\\
            \label{eq:murd}
            &\mu(t)
            = \Bigl(
            I_d,
            -H_p\lp t,\cdot,\nabla_xu(t,\cdot),\mu(t)\rp
            \Bigr){\#}m(t)
            &\text{ in } [0,T],\\
            \label{eq:CFurd}
            &u(T,x)
            =g(x,m(T))
            &\text{ in }
            \RR^d,\\
            \label{eq:CImrd}
            &m(0,x)
            =m_0(x)
            &\text{ in }
            \RR^d.
    \end{empheq}
\end{subequations}
Here, making a distinction between $\mu_{\aa}$ and $\mu_b$
is pointless since they coincide.
Therefore, we simply use the notation $\mu$.
For the system \eqref{eq:MFGCrd},
the Hamiltonian is defined as the Legendre transform of $L$,
\begin{equation}
    \label{eq:defHrd}
    H\lp t,x,p,\mu\rp
    =
    \sup_{\aa\in\RR^d}
    -p\cdot\aa-L\lp t,x,\aa,\mu\rp.
\end{equation}

\begin{definition}
    \label{def:sol}
    We say that
    $\lp u,m,\mu_{\aa},\mu_b\rp$
    is a solution to
    \eqref{eq:MFGCb}
    if
    \begin{itemize}
        \item
            $u\in C^{0,1}\lp [0,T]\times\RR^d;\RR\rp$
            %(resp. $C^{0,1}\lp [0,T]\times\TT^d_a;\RR\rp$)
            is a solution to the heat equation in the sense
            of distributions with a right-hand side equal to
            $(t,x)\mapsto f(t,x,m(t))- H\lp t,x,\nabla_xu,\mu(t)\rp$,
            and satisfies the terminal condition
            \eqref{eq:CFub},
            %(resp. \eqref{eq:CFuMono}),
        \item
            $m\in C^0\lp[0,T];\cP\lp\RR^d\rp\rp$
            %(resp. $C^0\lp[0,T];\cP\lp\TT^d_a\rp\rp$)
            is a solution to
            \eqref{eq:FPKb}
            %(resp. \eqref{eq:FPKMono})
            in the sense of distributions,
            and satisfies the initial condition
            \eqref{eq:CImb},
            %(resp. \eqref{eq:CImMono}),
        \item
            $\mu_{\aa}(t),\mu_b(t)\in\cP\lp\RR^d\times\RR^d\rp$
            %(resp. $\mu(t)\in\cP\lp\TT^d_a\times\RR^d\rp$)
            satisfy \eqref{eq:muaa} and \eqref{eq:mub}
            %(resp. satisfies \eqref{eq:muMono})
            for any $t\in[0,T]$.
    \end{itemize}
    We say that
    $\lp u,m,\mu\rp$
    is a solution to
    \eqref{eq:MFGCrd}
    %(resp. \eqref{eq:MFGCMono})
    if $u$ and $m$ respectively satisfy
    the first two points of the
    latter definition with \eqref{eq:MFGCrd}
    %(resp. \eqref{eq:MFGCMono})
    instead of \eqref{eq:MFGCb},
    and if $\mu(t)\in\cP\lp\RR^d\times\RR^d\rp$
    %(resp. $\cP\lp\TT^d_a\times\RR^d\rp$)
    satisfies \eqref{eq:murd}
    %(resp. \eqref{eq:muMono})
    for any $t\in[0,T]$.
\end{definition}

\subsection{Hypothesis}
\label{subsec:hypoMono}
%In what follows we state all the assumptions
%with $\RR^d$ as the state set.
%Later, when we need one of these assumptions
%while working with
%the system \eqref{eq:MFGCMono}
%which state set is $\TT^d_a$,
%we shall simply replace $\RR^d$
%by $\TT^d_a$ as the state set in the chosen assumption
%(note that we keep $\RR^d$ as the set of admissible controls).
The monotonicity assumption made in this paper
concerns the Lagrangian.
For this reason and the fact that sometimes
it may be hard to obtain an explicit form
of the Hamiltonian (like in the example
of paragraph \ref{subsec:abncrowd} below),
all the assumptions will be formulated in term of the Lagrangian
and never in term of the Hamiltonian.
In particular, working with the Lagrangian
gives more flexibility in the arguments of the proofs.
%The main advantage of considering only assumptions on the
%Lagrangian is that when constructing
%a new Lagrangian $L'$ from $L$,
%it is easy to check that the assumptions
%on $L$ hold for $L'$.
%In particular, the assumptions below are
%preserved by convex combination of Lagrangians,
%therefore we can easily construct approximating
%sequences of $L$ satisfying the same assumptions
%as $L$.

The constants entering the assumptions are
$C_0$ a positive constant,
$q\in(1,\infty)$ an exponent,
$q'=\frac{q}{q-1}$ its conjugate exponent,
and $\bb_0\in(0,1)$ a Hölder exponent. 
\begin{enumerate}[label=\bf{A\arabic*}]
    %\item
     %   \label{hypo:Lstrconv}
      %  {\color{red}$L$ is stricly convex with respect to $\aa$.}
    \item
        \label{hypo:Lreg}
        $L:[0,T]\times\RR^d\times\RR^d\times\cP\lp\RR^d\times\RR^d\rp
        \to\RR$ 
        is differentiable with respect to $\lp x,\aa\rp$;
        %with continuous derivatives;
        $L$ and its derivatives 
        are continuous on
        $[0,T]\times\RR^d\times\RR^d\times\cP_{\infty,R}\lp\RR^d\times\RR^d\rp$
        for any $R>0$;
        we recall that $\cP_{\infty,R}\lp\RR^d\times\RR^d\rp$
        is endowed with the weak* topology
        on measures;
        we use the notation $L_x$, $L_{\aa}$
        and $L_{(x,\aa)}$ for
        respectively the first-order derivatives
        of $L$ with respect to $x$, $\aa$
        and $(x,\aa)$
        .
    \item
        \label{hypo:Lbconv}
        The maximum in \eqref{eq:defHMono}
        is achieved at a unique $\aa\in\RR^d$.
        
        %When the drift is equal to the control,
        %\ref{hypo:Lbconv} is equivalent to assuming that $L$ is
        %strictly convex in $\aa$.
    \item
        \label{hypo:LMono}
        $L$ satisfies the following monotonicity condition,
        \begin{equation*}
            \int_{\RR^d\times\RR^d}
            \lp L\lp t, x,\a,\mu^1 \rp
            -L\lp t, x,\a,\mu^2\rp \rp
            d\lp \mu^1-\mu^2 \rp (x,\aa)
            \geq
            0.
        \end{equation*}
        for any $t\in[0,T]$,
        $\mu^1, \mu^2 \in
        \cP\lp\RR^d\times \RR^d \rp$.
    \item
        \label{hypo:Lcoer}
        $L(t,x,\aa,\mu)
        \geq
        C_0^{-1}|\aa|^{q'}
        -C_0\lp 1+\LL_{q'}\lp\mu\rp^{q'}\rp$,
        where $\LL_{q'}$ is defined in \eqref{eq:defL},
    \item
        \label{hypo:Lbound}
        $\labs L(t,x,\aa,\mu)\rabs
        \leq
        C_0\lp1+|\aa|^{q'}
        +\LL_{q'}\lp\mu\rp^{q'}\rp$,
        and
        $\labs L_x(t,x,\aa,\mu)\rabs
        \leq
        C_0\lp1+|\aa|^{q'}
        +\LL_{q'}\lp\mu\rp^{q'}\rp$,
    \item
        \label{hypo:Monogregx}
        $\int_{\RR^d}|x|^2dm^0(x)\leq C_0$,
        $\norm{m^0}{C^{\bb_0}}\leq C_0$,
        %$\norm{x}{L^2\lp m^0\rp}\leq C_0$,
        $\norm{f(t,\cdot,m)}{C^{1}}\leq C_0$,
        $\norm{g(\cdot,m)}{C^{2+\bb_0}}\leq C_0$,
        for any $t\in[0,T]$
        and $m\in\cP\lp\RR^d\rp$.
        \begin{comment}
    \item
        \label{hypo:Llipmu}
        For $R>0$, there exists
        a constant $C_R>0$ 
        such that
        \begin{equation*}
            \labs L_{(x,\aa)}\lp t,x,\aa,\mu^1\rp
            - L_{(x,\aa)}\lp t,x,\aa,\mu^2\rp\rabs
            \leq
            C_R\lp\int_{\RR^d}W_2\lp\mu^1_x,\mu^2_x\rp^2dm(x)\rp^{\frac12},
        \end{equation*}
        for $\lp t,x,\aa,m,\mu^i\rp$
        such that $\lp t,x,\aa\rp\in[0,T]\times\RR^d\times\RR^d$
        with $|\aa|\leq R$,
        $m\in\cP_2\lp\TT^d\rp$
        with $\lp\int_{\RR^d}|x|^2dm(x)\rp^{\frac12}\leq R$,
        $\mu^i\in\cP_m\lp\TT^d\times\RR^d\rp$
        with $\LL_{\infty}\lp\mu^i\rp\leq R$,
        and
        $\lp\mu^i_x\rp_{x\in\RR^d}$ is a familly of $\cP\lp\RR^d\rp$
        obtained by the disintegration theorem of measures applied to
        $\mu^i$ and $m$ (it is $m$-almost everywhere uniquely defined),
        $i=1,2$.
        \end{comment}
\end{enumerate}
Assumption \ref{hypo:LMono}
can be interpreted as a natural extension of
the Lasry-Lions monotonicity condition to MFGC.
%It has been introduced 
%in \cite{MR3112690}.
%The usual Lasry-Lions monotonicity condition
%to the coupling function and the terminal cost
%is named \ref{hypo:gMono} here,
%it is a very common assumption in MFG literature
%since existence and uniqueness of solution to 
%\eqref{eq:tradMFG} have been proved
%in \cite{Lions_video}
%assuming that $f$
%and $g$ are bounded,
%and some additional assumptions.
Roughly speaking,
the Lasry-Lions monotonicity condition
in the case of MFG without
interaction through controls,
translates the fact that the agents
have aversion for crowed regions of the state space.
In the case of MFGC,
the monotonicity condition implies
that the agents favor
moving in a direction opposite to the mainstream.
Such an assumption is adapted to models of
agents trading goods or financial assets.
Indeed in most of the models coming from economics or finance,
a buyer may prefer to buy when no one else is buying,
and conversely
a seller may prefer to sell when no one else is selling.

Assumptions
\ref{hypo:Lcoer} and \ref{hypo:Lbound}
imply that at least asymptotically 
when $\aa$ tends to infinity,
$L$ behaves like a power of $\aa$ of exponent
$q'$.
Under the monotonicity assumption
\ref{hypo:LMono},
uniqueness is in general easier
to obtain than existence.
For uniqueness, we assume that $f$ and $g$
are also monotonous,
this is the purpose of the following assumption.
%Moreover assuming 
%\ref{hypo:MonoLconvex}
%allows us to generalize the uniqueness
%result of \cite{MR3112690}
%from their version
%of \ref{hypo:LMono}
%(with a strict inequality when $\mu^1\neq\mu^2$)
%to the one presented here
%which in particular
%also covers the case of MFG without
%interaction through controls,
%unlike \cite{MR3112690}.

\begin{enumerate}[label=\bf{U}]
    %\item
     %   \label{hypo:MonoLconvex}
      %  The Lagrangian $L\lp t,x,\aa,\mu\rp$
       % is stricly convex with respect to $\aa$,
       % i.e.
       % \begin{equation*}
       %     L\lp t, x,\aa^2,\mu\rp
       %     -L\lp t,x,\aa^1,\mu\rp
       %     -\lp\a^2-\aa^1\rp
       %     \cdot L_{\aa}\lp t, x,\aa^1,\mu\rp
       %     >
       %     0,
       % \end{equation*}
       % for any $x\in\RR^d$,
       % $\mu\in\cP\lp\RR^d\times\RR^d\rp$,
       % and $\aa^1,\aa^2\in\RR^d$ such that
       % $\aa^1\neq\aa^2$.
    \item 
        \label{hypo:gMono}
        For $m^1,m^2\in \cP\lp \RR^d\rp$,
        and $t\in[0,T]$,
        assume that,
        \begin{align*}
            \int_{\RR^d}
            \lp f\lp t,x,m^1 \rp
            -f\lp t,x,m^2\rp \rp
            d\lp m^1-m^2 \rp (x)
            \geq
            0,
            \\
            \int_{\RR^d}
            \lp g\lp x,m^1 \rp
            -g\lp x,m^2\rp \rp
            d\lp m^1-m^2 \rp (x)
            \geq
            0.
        \end{align*}
\end{enumerate}
In fact, assuming that $f$ satisfies the inequality
in \ref{hypo:gMono}, implies that we can take $f=0$
up to replacing $L$ by $L+f$
and $H$ by $H-f$.
However,
since \ref{hypo:gMono}
is not assumed for proving the existence of solutions,
we have chosen to write this assumption
explicitly, and keeping $f\neq0$ is not pointless.

Let us now make assumptions on the drift
function $b$,
which concern the system \eqref{eq:MFGCb},
\begin{enumerate}[label=\bf{B\arabic*}]
    \item
        \label{hypo:binvert}
        There exists a function
        $\aa^*:[0,T]\times\RR^d\times\RR^d\times\cP\lp\RR^d\times\RR^d\rp
        \rightarrow\RR^d$
        such that for any
        $\lp t,x,\bt,\mu_{\aa}\rp
        \in[0,T]\times\RR^d\times\RR^d\times\cP\lp\RR^d\times\RR^d\rp$,
        \begin{equation*}
            \bt
            =
            b\lp t,x,\aa^*\lp t,x,\bt,\mu_b\rp,\mu_{\aa}\rp,
        \end{equation*}
        where $\mu_b$ is defined by
        $\mu_b
        =
        \Bigl[
        \lp x,\aa\rp
        \mapsto
        \lp x,
        b\lp t,x,\aa,\mu_{\aa}\rp\rp
        \Bigr]{\#}\mu_{\aa}$.
        Moreover $\aa^*$ is
        differentiable with respect to $x$ and $\bt$;
        $\aa^*$ and its derivatives 
        are continuous on
        $[0,T]\times\RR^d\times\RR^d\times\cP_{\infty,R}\lp\RR^d\times\RR^d\rp$
        for any $R>0$.
    %\item
     %   \label{hypo:Lbstrconv}
      %  {\color{red}$\bt\mapsto L\lp t,x,\aa^*\lp t,x,\bt,\mu_{b}\rp,\mu_{\aa}\rp$
      %      is stricly convex,
      %      for any 
      %  $\lp t,x,\mu_{\aa}\rp
      %  \in[0,T]\times\RR^d\times\cP\lp\RR^d\times\RR^d\rp$
      %  and $\mu_b$ defined by
      %  $\mu_b
      %  =
      %  \Bigl[
      %  \lp x,\aa\rp
      %  \mapsto
      %  \lp x,
      %  b\lp t,x,\aa,\mu_{\aa}\rp\rp
      %  \Bigr]{\#}\mu_{\aa}$.}
    \item
        \label{hypo:bbound}
        $b$ and $\aa^*$ satisfy
    \begin{align*}
        &\labs b\lp t,x,\aa,\mu_\aa\rp\rabs
        \leq
        C_0\lp 1+ |\aa|^{q_0}
        +\LL_{q'}\lp\mu_{\aa}\rp^{q_0}\rp,
        \\
        &\labs \aa^*\lp t,x,\bt,\mu_{b}\rp\rabs^{q_0}
        +\labs \aa^*_x\lp t,x,\bt,\mu_{b}\rp\rabs^{q_0}
        \leq
        C_0\lp 1+ |\bt|
        +\mathbf{1}_{q_0\neq 0}\LL_{\frac{q'}{q_0}}\lp\mu_{b}\rp\rp,
    \end{align*}
    for some exponent $q_0$ such that $0\leq q_0\leq q'$.
\end{enumerate}
Roughly speaking \ref{hypo:binvert} means that $b$ is invertible with respect
to $\aa$ in such a way that its inverse map can be expressed
in term of $\mu_b$ instead of $\mu_{\aa}$.
Conversely, if \ref{hypo:binvert} holds,
then for any
$\lp t,x,\aa,\mu_{b}\rp
\in[0,T]\times\RR^d\times\RR^d\times\cP\lp\RR^d\times\RR^d\rp$,
\begin{equation*}
    \aa
    =
    \aa^*\lp t,x,b\lp t,x,\aa,\mu_{\aa}\rp,\mu_{b}\rp.    
\end{equation*}
where $\mu_{\aa}$ is defined by
$\mu_{\aa}
=
\Bigl[
\lp x,\bt\rp
\mapsto
\lp x,
\aa^*\lp t,x,\bt,\mu_{b}\rp\rp
\Bigr]{\#}\mu_{b}$.
%{\color{red}
%Assumption \ref{hypo:Lbstrconv}
%is the counterpart of 
%\ref{hypo:Lstrconv}
%when the drift function
%differs from the control.
%}
The inequalities in \ref{hypo:bbound}
means that $|b|$ behaves asymptotically like a
power of $\aa$ with exponent
$q_0$, when $|\aa|$ is large.

To our knowledge,
such a general assumption on the class of drift functions
has not been made in the MFGC literature.

\subsection{Main results}
\label{subsec:mainMono}
The two main results in this work are
Theorems \ref{thm:UniMonob}
and \ref{thm:bpower} below,
which respectively state the existence
and uniqueness of the solution to \eqref{eq:MFGCb}.
\begin{theorem}
    \label{thm:UniMonob}
    Under assumptions
    \ref{hypo:Lreg}-\ref{hypo:LMono},
    \ref{hypo:gMono}
    and \ref{hypo:binvert}
    %and \ref{hypo:Lbstrconv},
    there is at most one solution
    to \eqref{eq:MFGCb}.
\end{theorem}
Uniqueness results for MFGC systems
with a monotonicity assumption
have been proved in
\cite{MR3112690} and \cite{MR3752669}.
In \cite{MR3112690},
uniqueness is proved when
the diffusion coefficient is equal
to $0$ and the drift is
equal to the control, i.e. $\nu=0$ and $b=\aa$.
In \cite{MR3752669} Section $4.6$,
the authors stated uniqueness
in the quadratic case.
Theorem
\ref{thm:UniMonob} is new in the sense
that it yields uniqueness for a large new class
of Lagrangians and drift functions.
Indeed, beside the monotonicity
assumption \ref{hypo:LMono} and \ref{hypo:gMono},
we only assume that $L$ satisfies \ref{hypo:Lreg} and
\ref{hypo:Lbconv}, and
%\ref{hypo:Lbstrconv}, and
that the drift $b$ is invertible
in the sense of \ref{hypo:binvert}.
\begin{theorem}
    \label{thm:bpower}
    Under assumptions
    \ref{hypo:Lreg}-\ref{hypo:Monogregx}
    and
    \ref{hypo:binvert}-\ref{hypo:bbound},
    there exists a solution to
    \eqref{eq:MFGCb}.
    %
    %Moreover, if $b=\aa$
    %the conclusion holds assuming only
    %\ref{hypo:Lreg}-\ref{hypo:Llipmu}.
\end{theorem}
The existence of solutions of the MFGC system
is in general much more demanding than
for MFG systems without interactions through the controls.
Under monotonicity assumptions similar to \ref{hypo:LMono},
existence has been proved in \cite{MR3752669}
Section $4.6$,
for quadratic and uniform convex Lagrangians
with a growth condition on the derivatives of
the Hamiltonian.
In \cite{MR3805247},
the existence of weak solutions
of the monotonous MFGC system is discussed with
a possibly degenerate diffusion operator,
under assumptions which are uniform with respect
to the joint law of states and controls.

Here, we prove existence of solutions of
the monotonous MFGC system for a large class
of Lagrangians and the drifts.
Namely, we assume that the Lagrangians and drifts
behave asymptotically like a power of $\aa$;
we allow them to have a growth in
the law of the controls
of at most the same order as the order of dependency upon $\aa$.
%the Lagrangians are strictly convex with respect to $\aa$.
%Assumption \ref{hypo:Lbconv} implies that 
%replacing the optimization problem in $\aa$
%by an optimization problem in the drift $b$
%preserves the convexity.

%The first step in our strategy for proving existence
%is to look for a priori estimates for the solutions
%of the MFGC systems and obtain compactness results
%to use a fixed point theorem.
Before starting the discussion on existence
of solutions to the MFGC systems
\eqref{eq:MFGCb}
and \eqref{eq:MFGCrd},
we introduce a new MFGC system
set in the torus,
so that the solutions should
have more compactness properties.
We define $\TT^d_a=\RR^d/\lp a\ZZ^d\rp$ the
$d$-dimensional torus of radius $a>0$.
Namely, we consider:
\begin{subequations}\label{eq:MFGCttda}
     \begin{empheq}{align}
    \label{eq:HJBttda}
    &-\ptt u(t,x)
    - \nu\Delta u(t,x)
    + H\lp t,x, \nabla_xu(t,x),\mu(t)\rp
    =
    f(t,x,m(t))
    &\text{ in }
    (0,T)\times\TT^d_a,\\
    \label{eq:FPKttda}
    & \ptt m(t,x)
    - \nu\Delta m(t,x)
    -\divo\lp H_p\lp t,x,\nabla_xu(t,x),\mu(t)\rp m\rp
    =
    0
    &\text{ in }
    (0,T)\times\TT^d_a,\\
    \label{eq:muttda}
    &\mu(t)
    = \Bigl(
    I_d,
    - H_p\lp t,\cdot,\nabla_xu(t,\cdot),\mu(t)\rp
    \Bigr){\#}m(t)
    &\text{ in } [0,T],\\
    \label{eq:CFuttda}
    &u(T,x)
    =
    g(x,m(T))
    &\text{ in }
    \TT^d_a,\\
    \label{eq:CImttda}
    &m(0,x)
    =
    m_0(x)
    &\text{ in }
    \TT^d_a.
    \end{empheq}
\end{subequations}
All the assumptions in paragraph \ref{subsec:hypoMono}
are stated in $\RR^d$.
When considering that
$L:[0,T]\times\TT^d_a\times\RR^d\times\cP\lp\TT^d_a\times\RR^d\rp\to\RR$
(like in \eqref{eq:MFGCttda})
satisfies one of those assumptions,
we shall simply replace $\RR^d$
by $\TT^d_a$ as the state set in the chosen assumption.
%(note that we keep $\RR^d$ as the set of admissible controls).

The fixed point satisfied by the joint law
of states and controls,
namely \eqref{eq:muaa}-\eqref{eq:mub},
\eqref{eq:murd} or \eqref{eq:muttda},
may be a difficult issue for MFGC systems.
Here, using mainly the monotonicity
assumption \ref{hypo:LMono}
and the compactness of the state space
of \eqref{eq:MFGCttda},
we prove
in Section \ref{sec:MonoFPV}
the following lemma
which states well-posedness for the fixed
point \eqref{eq:muttda},
and ensures continuity with respect to time.
\begin{lemma}
    \label{lem:FPmucont}
    Assume
    \ref{hypo:Lreg}-\ref{hypo:Lbound}.
    %\ref{hypo:Lstrconv}-\ref{hypo:Lbound} and
    Let
    $p\in C^0\lp[0,T]\times\TT^d_a;\RR^d\rp$
    and
    $m\in C^0\lp[0,T];\cP\lp\TT^d_a\rp\rp$
    be such that
    $t\mapsto p(t,\cdot)$
    is continuous with respect to the topology
    of the local uniform convergence
    and $m(t)$ admits a finite second
    moment uniformly bounded with respect to $t\in[0,T]$.
    For any $t\in[0,T]$,
    there exists a unique $\mu(t)\in\cP\lp\TT^d_a\times\RR^d\rp$
    such that
    \begin{equation*}
        \mu(t)
        =
        \lp I_d,
        -H_p\lp t,\cdot,p(t,\cdot),\mu(t)\rp\rp\# m(t).
    \end{equation*}
    Moreover, the map $t\mapsto \mu(t)$
    is continuous where $\cP\lp\TT^d_a\times\RR^d\rp$ is
    equipped with the weak* topology.
\end{lemma}

The next step in our strategy for proving existence
is to look for a priori estimates for the solutions
of the MFGC systems and obtain compactness results
to use a fixed point theorem.
In section \ref{sec:aprioriMono},
we prove the a priori estimates
stated in the following lemma
for solutions to \eqref{eq:MFGCttda}.
\begin{lemma}
    \label{lem:Monoapriorith1}
    Assume
    \ref{hypo:Lreg}-\ref{hypo:Monogregx}.
    %\ref{hypo:Lstrconv}-\ref{hypo:Monogregx}.
    If $(u,m,\mu)$ is a solution to \eqref{eq:MFGCttda},
    then
    $\norminf{u}$,
    $\norminf{\nabla_xu}$ and
    $\displaystyle{\sup_{t\in[0,T]}} \LL_{\infty}\lp\mu(t)\rp$
    are uniformly bounded by a constant
    independent of $a$.
\end{lemma}
Let us mention that the a priori estimates of
Lemma \ref{lem:Monoapriorith1} rely on the monotonicity
assumption on $L$ and a Bernstein method
introduced in \cite{2019arXiv190411292K}.
To our knowledge, these are the first results
in the literature of MFGC which use
the monotonicity assumption
for getting a priori estimates.
They are 
the key ingredients of the proof of the
existence of solutions to \eqref{eq:MFGCttda}
in the following theorem,
proved in paragraph \ref{subsec:Exittda}.
\begin{theorem}
    \label{thm:ExiMonottda}
    Under assumptions
    \ref{hypo:Lreg}-\ref{hypo:Monogregx},
    %\ref{hypo:Lstrconv}-\ref{hypo:Llipmu},
    there exists a solution to system
    \eqref{eq:MFGCttda}.
\end{theorem}
Therefore,
for any $a>0$ we can construct a solution
to \eqref{eq:MFGCttda} which satisfies
uniform estimates with respect to $a$.
This allows us to construct a compact sequence of
approximating solutions to \eqref{eq:MFGCrd}.
Passing to the limit for a subsequence allows us 
to generalize the conclusion
of Theorem \ref{thm:ExiMonottda}
to system \eqref{eq:MFGCrd}.
This leads to the following theorem
proved in paragraph \ref{subsec:rd}.
\begin{theorem}
    \label{thm:Exird}
    Under assumptions
    \ref{hypo:Lreg}-\ref{hypo:Monogregx},
    %\ref{hypo:Lstrconv}-\ref{hypo:Llipmu},
    there exists a solution to
    \eqref{eq:MFGCrd}.
\end{theorem}
Uniqueness relies on the monotonicity
assumptions \ref{hypo:LMono}
and \ref{hypo:gMono},
the following theorem is proved
in paragraph \ref{subsec:UniMono}.
\begin{theorem}
    \label{thm:UniMono}
    Under assumptions
    \ref{hypo:Lreg}-\ref{hypo:LMono}
    %\ref{hypo:Lstrconv}-\ref{hypo:LMono}
    and
    \ref{hypo:gMono},
    there is at most one solution
    to
    \eqref{eq:MFGCrd}
    or~\eqref{eq:MFGCttda}.
\end{theorem}
The idea to pass from \eqref{eq:MFGCrd} to \eqref{eq:MFGCb},
is to change the optimization problem
in $\aa$ into a new optimization problem
expressed in term of $b$.
In paragraph \ref{subsec:Exib},
we prove the equivalence between the
solutions of these two
optimization problems.
A first existence results for
\eqref{eq:MFGCb} is stated in
Corollary \ref{thm:ExiMonob}
which uses this equivalence.
Theorem \ref{thm:bpower} is a consequence
of Corollary \ref{thm:ExiMonob}
with more tractable assumptions.
Let us mention that for proving Theorem
\ref{thm:bpower}, 
the structure of the Lagrangian
should be invariant when passing from
one optimization problem to the other.
In particular,
one may figure out that
the assumptions on the Lagrangian
behaving asymptotically like a power of $\aa$
are preserved under our assumptions
on the drift function $b$.

Finally, Theorem \ref{thm:UniMonob} is a consequence
of Theorem \ref{thm:UniMono} and the above-mentionned
equivalence between the two optimization problems.

\begin{remark}
    \begin{enumerate}[\roman*)]
        \item
            If the Lagrangian admits the following
            form,
            \begin{equation*}
                L\lp t,x,\aa,\mu\rp
                =
                L^0\lp t,x,\aa\rp
                +L^1\lp t,\mu\rp,
            \end{equation*}
            we say that the Lagrangian is separated.
            Then \ref{hypo:LMono}
            is automatically satisfied
            since the left-hand side
            of the inequality is identically equal to $0$.
            In this case,
            the assumptions on $L$ are satisfied 
            if $L^0$
            behaves asymptotically like a power of $\aa$ 
            of exponenent $q'$,
            and $L^1$
            at most involves $\Lambda_{q_0}(\mu)^{q'}$.

            Here, we do not provide an explicit application in which
            the Lagrangian is separated,
            however this is a general hypothesis in
            the MFGC literature.
            Therefore, our framework in the present paper
            can be seen as an extension of the case
            when $L$ is separated.
        \item
All our assumptions are uniform with respect to the state
variable $x$. In particular, we restrain from considering more general
functions $f$ and $g$ since this topic has been
investigated in the literature devoted to
MFG systems without interaction through controls;
we believe that the same tools can be applied
to the present case, and that our results may be extended so.

\item
We did not address the case without diffusion,
i.e. $\nu=0$.
However, all the a priori estimates
of Sections \ref{sec:MonoFPV}
and \ref{sec:aprioriMono}
are uniform with respect to $\nu$.
Here, the diffusion is used
to easily obtain compactness results
which are central for proving our existence results
since the proofs rely on a fixed point theorem
and approximating sequences of solutions.
Using weaker topological spaces
and tools from the literature
devoted to weak solutions of systems
of MFGs without interaction through controls,
we believe that we can extend our results
to weak solutions to MFGC systems
without diffusion or
with possibly degenerate diffusion operators.
We plan to address this question in forthcoming works.
\end{enumerate}
\end{remark}

\subsection*{General outline}
\label{subsec:outlineMono}
The present work aims at proving Theorems
\ref{thm:UniMonob} and \ref{thm:bpower}.
We list below the main steps of our analysis
to make it easier for the reader to understand
the structure of the proofs.
\begin{enumerate}[label=\bf{\Roman*}]
    \item
        \label{step:FPmu}
        We solve the fixed
        point \eqref{eq:muttda}
        in $\mu$,
        which proves
        Lemma \ref{lem:FPmucont},
        in three steps:
        \begin{enumerate}[label=\bf{I.\alph*}]
            \item
                \label{step:FPmuapriori}
                in Lemma \ref{lem:apriorimu}
                we state a priori estimates
                for a solution of
                \eqref{eq:muttda};
            \item
                \label{step:FPmugen}
                using the Leray-Schauder fixed point theorm
                (Theorem \ref{thm:Leray}),
                we solve the fixed point \eqref{eq:muttda}
                at any time $t\in[0,T]$,
                in Lemma \ref{lem:FPmuLS};
            \item
                \label{step:FPmucont}
                we prove that the fixed point $\mu(t)$
                defined at any $t\in[0,T]$ in step
                \ref{step:FPmugen},
                is continuous with respect to time
                (Lemma \ref{lem:contmu});
                this implies
                lemma \ref{lem:FPmucont}.
        \end{enumerate}
    \item
        \label{step:ttda}
        We prove the existence of a solution 
        to \eqref{eq:MFGCttda},
        stated in Theorem \ref{thm:ExiMonottda},
        in two steps:
        \begin{enumerate}[label=\bf{II.\alph*}]
            \item
                \label{step:ttdaapriori}
                we obtain a priori estimates
                for solutions to \eqref{eq:MFGCttda}
                (Lemmas \ref{lem:Monoapriorith1}
                and \ref{lem:Monoapriori});
            \item
                \label{step:ttdaexi}
                in paragraph \ref{subsec:Exittda},
                we use Leray-Schauder fixed point theorem
                (Theorem \ref{thm:Leray})
                and the estimates
                of step \ref{step:ttdaapriori}
                to conclude.
        \end{enumerate}
    \item
        \label{step:MFGCrd}
        We prove existence and uniqueness of the solution
        to \eqref{eq:MFGCrd}
        (Theorems \ref{thm:Exird} and \ref{thm:UniMono}):
        \begin{enumerate}[label=\bf{III.\alph*}]
            \item
                \label{step:MFGCrdexi}
                the proof of 
                Theorem \ref{thm:Exird} is given
                in paragraph \ref{subsec:rd};
                %proving Theorem \ref{thm:Exird}, i.e.
                %constructing a solution to \eqref{eq:MFGCrd},
                %by passing to the limit in a subsequence
                %of approximate solutions contructed
                %using solutions to \eqref{eq:MFGCttda},
            \item
                \label{step:MFGCrduni}
                the proof of 
                Theorem \ref{thm:UniMono} is given
                in paragraph \ref{subsec:UniMono};
                %addressing the 
                %uniqueness of the solution to \eqref{eq:MFGCrd};
        \end{enumerate}
    \item
        \label{step:MFGCb}
        The proof of existence and uniqueness of the solution
        to \eqref{eq:MFGCb}
        (Theorems \ref{thm:UniMonob} and \ref{thm:bpower})
        is given in paragraph \ref{subsec:Exib}.
\end{enumerate}

\subsection*{Contribution}
An important novelty in the present work comes from
the assumptions we are considering.
On the one hand,
we consider a general class of
monotonous Lagrangians
which behave asymptotically like
a power of $\aa$
with any exponent in $(1,\infty)$ (while most
of the results in the literature only
address the quadratic case with
uniformly convex Lagrangian);
they may depend on moments
of $\mu_{\aa}$ at most of the same
order as the above-mentioned
exponent of $L$ in $\aa$;
%their dependency upon $\mu_{\aa}$
%may have the same order of magnitude as
%their dependency upon $\aa$;
we do not require them to
depend separately on
$\lp x,\aa\rp$ and $\mu_{\aa}$.
On the other hand,
the drift functions are also general
since we allow them to behave
like power functions
and to be not separated too.
See the assumptions in paragraph
\ref{subsec:hypoMono} for more details.

Moreover, 
most contributions focus on
MFG systems stated on $\TT^d$
for simplicity.
Here, we introduce a method to
extend an existence result
for a MFGC system
stated on the torus
to its counterpart
on the whole Euclidean space.
In particular, this method
holds for MFG system without interaction through controls
and the proof becomes easier.
See paragraph \ref{subsec:rd}.
We also introduce a method
to extend the well-posedness
of MFGC (or MFG) systems
to general drift functions,
see paragraph \ref{subsec:Exib}.
We would like to insist on the fact that
our techniques are designed
in order to preserve
the structure of the Lagrangian
when passing from one
setting to another.
Here, namely it preserves
the monotonicity assumption \ref{hypo:LMono}.
Furthermore, these methods apply
to the conclusions of \cite{2019arXiv190411292K}
and consequently generalize them.

\subsection{Properties of the Lagrangian and
    the Hamiltonian in \eqref{eq:MFGCrd} and \eqref{eq:MFGCttda}}
Here, we write the results and the proofs for
the Lagrangian and Hamiltonian
involved in system \eqref{eq:MFGCrd}.
However, none of the arguments below is
specific to the domain $\RR^d$,
therefore the conclusions hold
and the proofs can be repeated
for the Lagrangian and Hamiltonian
involved in \eqref{eq:MFGCttda}.

We start by proving that under the assumptions
of paragraph \ref{subsec:hypoMono},
when $b=\aa$, $L$ is strictly convex.
\begin{lemma}
    \label{lem:Lstrconv}
    If $L$ is coercive
    and differentiable with respect to $\aa$,
    and $b=\aa$,
    assuming that
    $L$ is strictly convex
    is equivalent to \ref{hypo:Lbconv}.
\end{lemma}

\begin{proof}
    If $L$ is stricly convex and coercive,
    it is straightforward to check \ref{hypo:Lbconv}.

    Conversely, we
    take $\lp t,x,\mu\rp \in[0,T]\times\RR^d\times\cP\lp\RR^d\times\RR^d\rp$.
    We set $\ell(\aa)=L\lp t,x,\aa,\mu\rp$.
    It is sufficient to prove that $\ell$ is strictly convex.

    \emph{First step: proving that $\ell$ is convex.}

    We define $\ell^{**}$ as the biconjugate of $\ell$,
    $\ell^{**}$ is in particular the Legendre transform
    of $H\lp t,x,\cdot,\mu\rp$.
    The map $\ell^{**}$ is convex and continuous
    since $\ell$ is coercive,
    and it satisfies $\ell^{**}\leq \ell$.
    In what follows, we will prove that $\ell^{**}=\ell$. 

    We assume by contradiction that $\ell^{**}\neq \ell$:
    there exists $\aa^0\in\RR^d$ such that 
    $\ell^{**}\lp\aa^0\rp<\ell\lp\aa^0\rp$.
    We recall that $\ell$ and $\ell^{**}$ admit the same convex envelope,
    therefore by Carathéorthéodory's theorem,
    there exists
    $\lp\aa^i\rp_{1\leq i\leq d+1}\in\lp\RR^d\rp^{d+1}$
    and
    $\lp\ll^i\rp_{1\leq i\leq d+1}\in\lp\RR_+\rp^{d+1}$
    such that
    \begin{equation*}
        \aa^0
        =
        \sum_{i=1}^{d+1}\ll^i\aa^i,
        \;\;
        \ell^{**}\lp\aa^0\rp
        =
        \sum_{i=1}^{d+1}\ll^i\ell\lp\aa^i\rp,
        \;\;\text{ and }\;\;
        \sum_{i=1}^{d+1}\ll^i
        =
        1.
    \end{equation*}
    Using the inequality $\ell^{**}\leq \ell$,
    we obtain that
    \begin{equation*}
        \ell^{**}(\aa)
        =
        \sum_{i=1}^{d+1}\ll^i\ell\lp\aa^i\rp
        \geq
        \sum_{i=1}^{d+1}\ll^i\ell^{**}\lp\aa^i\rp.
    \end{equation*}
    This inequality is in fact an equality since $\ell^{**}$ is convex,
    which implies that $\ell^{**}\lp\aa^i\rp=\ell\lp\aa^i\rp$
    for any $i\in\lc1,2,\dots,d+1\rc$.
    Take $p\in\partial \ell^{**}\lp\aa^0\rp$,
    where $\partial \ell^{**}\lp\aa^0\rp$ is the subdifferential
    of $\ell^{**}$ at $\aa^0$.
    For $i\in\lc1,\dots,d+1\rc$,
    this implies
    \begin{equation*}
        \ell\lp \aa^i\rp
        =
        \ell^{**}\lp \aa^i\rp
        \geq
        \ell^{**}\lp \aa^0\rp
        +p\cdot\lp\aa^i-\aa^0\rp.
    \end{equation*}
    Multiplying the latter inequality by $\ll^i$
    and taking the sum over $i$,
    yield that it is in fact an equality.
    Then, 
    it is straightforward to check that
    $p\in\partial \ell^{**}\lp\aa^i\rp$
    for any $i$;
    this implies that $p=\nabla_{\aa}\ell\lp\aa^i\rp$,
    since $\ell^{**}\lp\aa^i\rp=\ell\lp\aa^i\rp$ and
    $\ell$ is differentiable with respect to $\aa$.
    The maximum in the definition of $H(t,x,-p,\mu)$ is achieved
    at any $\aa^i$, this is a contracdition
    with \ref{hypo:Lbconv}.
    Therefore $\ell=\ell^{**}$ and $\ell$
    is convex.

    \emph{Second step: $\ell$ is striclty convex.}
    
    By definition of the subdifferential of a convex function,
    $\aa\in\RR^d$ achieves the maximum in the definition of
    $H\lp t,x,-\nabla_{\aa}\ell\lp\aa\rp,\mu\rp$.
    Using the fact that this maximum is unique by \ref{hypo:Lbconv},
    we obtain the strict convexity of $\ell$,
    and the one of $L$ with respect to $\aa$.
\end{proof}

In paragraph \ref{subsec:hypoMono},
we assume that $L$
behaves at infinity as a
power of $\aa$
of exponent $q'$.
Because of the conjugacy relation between
$L$ and $H$, it implies
that $H$ behaves at infinity like
a power of $p$ of exponent $q$.
\begin{lemma}
    \label{lem:Hbound}
    Under assumptions
    \ref{hypo:Lreg},
    \ref{hypo:Lbconv},
    \ref{hypo:Lcoer} and
    \ref{hypo:Lbound},
    the map $H$, defined in \ref{eq:defHrd},
    is differentiable with respect to $x$ and $p$,
    $H$ and its derivatives 
    are continuous on
    $[0,T]\times\RR^d\times\RR^d\times\cP_{\infty,R}\lp\RR^d\times\RR^d\rp$
    for any $R>0$.
    Moreover there exists $\Ct_0>0$ a constant which only depends
    on $C_0$ and $q$ such that
    \begin{align}
        \label{eq:Hpbound}
        \labs H_p\lp t,x,p,\mu\rp\rabs
        &\leq
        \Ct_0\lp 1+|p|^{q-1}
        +\LL_{q'}\lp\mu\rp\rp,
        \\
        \label{eq:Hbound}
        \labs H\lp t,x,p,\mu\rp\rabs
        &\leq
        \Ct_0\lp 1+|p|^{q}
        +\LL_{q'}\lp\mu\rp^{q'}\rp,
        \\
        \label{eq:Hcoer}
        p\cdot H_p\lp t,x,p,\mu\rp
        -H\lp t,x,p,\mu\rp
        &\geq
        \Ct_0^{-1}|p|^q
        -\Ct_0\lp1
        +\LL_{q'}\lp\mu\rp^{q'}\rp,
        \\
        \label{eq:Hxbound}
        \labs H_x\lp t,x,p,\mu\rp\rabs
        &\leq
        \Ct_0\lp 1+|p|^{q}
        +\LL_{q'}\lp\mu\rp^{q'}\rp,
    \end{align}
    for any $(t,x)\in[0,T]\times\RR^d$,
    $p\in\RR^d$ and $\mu\in\cP\lp\RR^d\times\RR^d\rp$.
\end{lemma}
Up to replacing $C_0$ with $\max(C_0,\Ct_0)$,
we can assume that the inequalities in Lemma
\ref{lem:Hbound} are satisfied with $C_0$
instead of $\Ct_0$.

Let us notice that
it is possible and not more difficult
to extend the results stated
in Lemma \ref{lem:Hbound} to the
Hamiltonian used in \eqref{eq:MFGCb}
and defined in \eqref{eq:defHMono},
however we will not have any use
of such results in the present paper.
\begin{proof} 
    \emph{First step: differentiability of $H$ in $p$,
    and continuity of $H$ and $H_p$.}

    For $(t,x,\mu)\in[0,T]\times\RR^d\times\cP\lp\RR^d\times\RR^d\rp$, 
    the map $\aa\mapsto L\lp t,x,\aa,\mu\rp$
    is stricly convex by Lemma \ref{lem:Lstrconv}
    and coercive by \ref{hypo:Lcoer};
    Theorem $26.6$
    in \cite{MR0274683}
    implies that $H$
    is differentiable with respect to $p$,
    the map
    $\aa\mapsto -L_{\aa}\lp t,x,\aa,\mu\rp$
    is invertible;
    its iverse map is 
    $p\mapsto -H_{p}\lp t,x,p,\mu\rp$
    by \cite{MR0274683}
    Theorem $26.5$.
    Theorem $26.6$
    in \cite{MR0274683}
    also implies that the maximum
    in \ref{eq:defHrd} is achieved
    by a unique control given
    by $-H_p\lp t,x,p,\mu\rp$.
    In the next step,
    we prove \ref{eq:Hpbound}
    which implies that $H_p$
    maps the bounded subsets of
    $[0,T]\times\RR^d\times\RR^d\times\cP_{\infty,R}\lp\RR^d\times\RR^d\rp$
    for $R>0$
    into relatively compact subspaces
    of $\RR^d$;
    we recall that $L_{\aa}$
    is continuous on
    $[0,T]\times\RR^d\times\RR^d\times\cP_{\infty,R}\lp\RR^d\times\RR^d\rp$;
    therefore $H_p$ is likewise continuous on the same space.
    We recall that $H$ satisfies
    \begin{equation*}
        H(t,x,p,\mu)
        =
        p\cdot H_p\lp t,x,p,\mu\rp
        -L\lp t,x, -H_p\lp t,x,p,\mu\rp,\mu\rp,
    \end{equation*}
    therefore $H$ is also continuous on the same spaces.

    %Since $L$ is strictly convex,
    %$H$ is differentiable with respect to $p$
    %and there exists a unique
    %$\aaa(t,x,p,\mu)$ 
    %which achieves the maximum in \eqref{eq:defHrd}
    %and which is given by
    %$\aaa=-H_p$.
    %Since $L$ is continuous,
    %the Maximum Theorem states that
    %$\aaa$ depends continuously on its arguments,
    %so does $H_p$.
     
    \emph{Second step: proving the first three inequalities of the Lemma.}
    Using the growth assumptions on $L$, we first prove
    \eqref{eq:Hpbound}.
    On the one hand we have that
    \begin{equation*}
        H\lp t,x,p,\mu\rp
        \geq
        -L\lp t,x,0,\mu\rp
        \geq
        -C_0\lp 1+\LL_{q'}(\mu)^{q'}\rp,
    \end{equation*}
    by \ref{hypo:Lcoer} and
    the condition of optimality in \eqref{eq:defHrd}.
    On the other hand,
    \ref{hypo:Lbound},
    the fact that $-H_p\lp t,x,p,\mu\rp$
    satisfies the optimality
    condition in \eqref{eq:defHrd},
    and the Young inequality
    $y\cdot z\leq
    \frac{|y|^{q}}{q} 
    +\frac{|z|^{q'}}{q'}$
    for $y,z\in\RR^d$,
    yield that,
    \begin{align*}
        H\lp t,x,p,\mu\rp
        &=
        p\cdot H_p\lp t,x,p,\mu\rp
        -L\lp t,x,-H_p\lp t,x,p,\mu\rp,\mu\rp
        \\
        &\leq
        \frac1{q'C_0}\labs H_p\lp t,x,p,\mu\rp\rabs^{q'}
        +\frac{C_0^{\frac{q}{q'}}}{q}\labs p\rabs^{q}
        -C_0^{-1}\labs H_p\lp t,x,p,\mu\rp\rabs^{q'}
        +C_0\lp 1+\LL_{q'}(\mu)^{q'}\rp
        \\
        &\leq
        -\frac1{qC_0}\labs H_p\lp t,x,p,\mu\rp\rabs^{q'}
        +\frac{C_0^{\frac{q}{q'}}}{q}\labs p\rabs^{q}
        +C_0\lp 1+\LL_{q'}(\mu)^{q'}\rp.
    \end{align*}
    Therefore, using the latter two chains of inequalities,
    and the fact that $\frac{q}{q'}=q-1$,
    we obtain that,
    \begin{equation}
        \label{eq:Hpq'bound}
        \frac1{qC_0}\labs H_p\lp t,x,p,\mu\rp\rabs^{q'}
        \leq
        \frac{C_0^{q-1}}{q}\labs p\rabs^{q}
        +2C_0\lp 1+\LL_{q'}(\mu)^{q'}\rp.
    \end{equation}
    This and the inequality
    $|y+z|^{\frac1{q'}}
    \leq
    |y|^{\frac1{q'}}
    +|z|^{\frac1{q'}}$
    for $y,z\in\RR$,
    imply that
    \begin{equation*}
        \labs H_p\lp t,x,p,\mu\rp\rabs
        \leq
        C_0^{q-1}\labs p\rabs^{q-1}
        +\lp2qC_0^2\rp^{\frac1{q'}}
        \lp 1+\LL_{q'}(\mu)\rp.
    \end{equation*}
    From \ref{hypo:Lbound} and  
    \eqref{eq:Hpq'bound},
    we obtain that,
    \begin{align*}
        \labs H\lp t,x,p,\mu\rp\rabs
        &=
        \labs
        p\cdot H_p\lp t,x,p,\mu\rp
        -L\lp t,x,-H_p\lp t,x,p,\mu\rp,\mu\rp\rabs
        \\
        &\leq
        \frac{|p|^q}q
        +\frac{\labs H_p\lp t,x,p,\mu\rp\rabs^{q'}}{q'}
        +C_0\lp 1+
        \labs H_p\lp t,x,p,\mu\rp\rabs^{q'}
        +\LL_{q'}(\mu)^{q'}\rp
        \\
        &\leq
        \lp\frac1q+\frac{C_0^q}{q'}\rp|p|^q
        +C_0\lp1+2qC_0\lp\frac1{q'}+C_0\rp\rp
        \lp1+\LL_{q'}(\mu)^{q'}\rp.
    \end{align*}
    We still have to prove \eqref{eq:Hcoer}.
    Let $\ee$ be a positive constant 
    depending only on $C_0$
    such that $\ee-C_0\ee^{q'}\geq\frac{\ee}2$,
    by the optimality condition in \eqref{eq:defHrd} used
    with $\aa=-\ee|p|^{q-2}p$, we have,
    \begin{align*}
        H\lp t,x,p,\mu\rp
        &\geq
        \ee |p|^q
        -L\lp t,x, -\ee|p|^{q-2}p,\mu\rp
        \\
        &\geq
        \ee |p|^q
        -C_0\lp 1+\ee^{q'}|p|^{(q-1)q'}+\LL_{q'}\lp\mu\rp^{q'}\rp
        \\
        &\geq
        \frac{\ee}2 |p|^q
        -C_0\lp 1+\LL_{q'}\lp\mu\rp^{q'}\rp.
    \end{align*}
    Then from \ref{hypo:Lbound}, 
    \begin{align*}
        H\lp t,x,p,\mu\rp
        &=
        p\cdot H_p\lp t,x,p,\mu\rp
        -L\lp t,x,-H_p\lp t,x,p,\mu\rp,\mu\rp
        \\
        &\leq
        \frac{\ee}4 |p|^q
        +\frac{\lp4q\ee^{-1}\rp^{\frac{q'}q}}{q'}
        \labs H_p\lp t,x,p,\mu\rp\rabs^{q'}
        +C_0\lp 1+
        \labs H_p\lp t,x,p,\mu\rp\rabs^{q'}
        +\LL_{q'}(\mu)^{q'}\rp.
    \end{align*}
    Combining the latter two chains of inequalities,
    there exists $C$ a positive constant depending only
    on $C_0$ such that
    \begin{equation*}
        \labs H_p\lp t,x,p,\mu\rp\rabs^{q'}
        \geq
        C^{-1}|p|^q
        -C\lp 1+ \LL_{q'}\lp\mu\rp^{q'}\rp.
    \end{equation*}
    This and \ref{hypo:Lcoer} yield that
    \begin{align*}
        p\cdot H_p\lp t,x,p,\mu\rp
        -H\lp t,x,p,\mu\rp
        &=
        L\lp t,x,-H_p\lp t,x,p,\mu\rp,\mu\rp
        \\
        &\geq
        C_0^{-1}
        \labs H_p\lp t,x,p,\mu\rp\rabs^{q'}
        -C_0\lp 1+ \LL_{q'}\lp\mu\rp^{q'}\rp
        \\
        &\geq
        \lp C_0C\rp^{-1}|p|^q
        -\lp C_0+C_0^{-1}C\rp\lp 1+\LL_{q'}(\mu)^{q'}\rp.
    \end{align*}
    
    %\emph{Third step:
     %   differentiability of $H$ in $x$ and the last inequality.}
    \emph{Third step:
        the smothness properties and the last inequality.}
    
    %This implies that
    %$(t,x,p,\mu)\mapsto
    %L_x\lp t,x,-H_p\lp t,x,p,\mu\rp,\mu\rp$
    %{\color{red}is continuous on the latter subsets}.
    From \eqref{eq:Hpbound},
    %Moreover the second step implies that
    $-H_p(t,x,p,\mu)$ is locally
    uniformly bounded,
    therefore we can reduce the set of admissible
    controls $\aa$ in \eqref{eq:defHrd} from
    $\RR^d$ to a compact subset of $\RR^d$.
    Within this framework, the envelop theorem
    states that $H$ is differentiable in $x$
    and its derivatives are defined by,
    \begin{equation*}
        H_x\lp t,x,p,\mu\rp
        =
        -L_x\lp t,x,-H_p\lp t,x,p,\mu\rp,\mu\rp.
    \end{equation*}
    The continuity property of $H_x$
    relies on the ones of $L_x$ and $H_p$.
    %is continuous in all its arguments.
    Moreover,
    from \ref{hypo:Lbound} and \eqref{eq:Hpq'bound},
    we obtain
    \begin{equation*}
        \labs H_x\lp t,x,p,\mu\rp\rabs
        \leq
        C_0^{q+1}|p|^{q}
        +C_0\lp1+2qC_0^2\rp
        \lp 1+\LL_{q'}\lp\mu\rp^{q'}\rp.
    \end{equation*}
    %Then, for if $(x,p,\mu)$ lie
    %in a compact subset $K_x\times K_p\times K_{\mu}$
    %of $\RR^d\times\RR^d\times\cP\lp\RR^d\times\RR^d\rp$
    %(where $\cP\lp\RR^d\times\RR^d\rp$ is endowed
    %with the weak* topology),
    %we can reduce the set of admissible
    %controls $\aa$ in \eqref{eq:defHrd} from
    %$\RR^d$ to a compact subset $K_{\aa}$ of $\RR^d$.
    %On $[0,T]\times K_x\times K_{\aa}\times K_{\mu}$,
    %$L$ is uniform convex by \ref{hypo:Lreg},
    %therefore $H_p$ is 
    This concludes the proof.
\end{proof}

\section{Applications}
\label{sec:appli}
\subsection{Exhaustible ressource model}
\label{subsec:Monoexhau}
This model is often referred to
as Bertrand and Cournot competition model for exhaustible ressources,
introduced in the independent works
of Cournot \cite{cournot1838recherches} and Bertrand \cite{Bertrand};
its mean field game version in dimension one was introduced in
\cite{MR2762362} and numerically analyzed in
\cite{MR3359708};
for theoretical results see
\cite{2019arXiv190205461F,MR3755719,MR4064472,MR3888969}.
%For a generalization to higher dimensions, which means that producers are selling
%differents ressources which prices and productions are dependent from each other,
%see \cite{2019arXiv190205461F}.
We consider a continuum of producers selling exhaustible ressources.
The production of a representative agent
at time $t\in[0,T]$ is denoted by $q_t\geq 0$;
the agents differ in their production capacity
$X_t\in\RR$ (the state variable),
that satifies,
\begin{equation*}
    dX_t
    =
    -q_tdt
    +\sqrt{2\nu} dW_t,
\end{equation*}
where $\nu>0$ and $W$ is a $d$-dimentional Brownian motion.
%and $\sigma\in \RR^{d\times d}$ is a positive definite matrix.
Each producer is selling a different ressource
and has her own consumers.
However, the ressources are substitutable
and any consumer may change her mind
and buy from a competitor
depending on the degree of competition
in the game (which stands for $\ee$ in the
linear demand case below for instance).
Therefore the selling price per unit of ressource
that a producer can make when
she sales $q$ units of ressource,
depends naturally on $q$
and on the quantity produced by the other
agents.
The price satisfies a
supply-demand relationship,
and is given by
$P\lp q,\qo\rp$,
where $\qo$ is the accumulated demand
which depends on 
the overall distribution of productions
of the agents.
%We consider the admissible controls to be markovian,
%i.e. $\aa_t=\aa(t,X_t)$.
A producer tries to maximize her profit,
or equivalently to minimize the following
quantity,
\begin{equation*}
    \EE\lb\int_0^T -P(q_t,\qo_t)\cdot q_tdt
    +g\lp X_T\rp\rb,
\end{equation*}
where $g$ is a terminal cost
which often penalizes the producers
who have non-zero production capacity at the end of the game.
In the Cournot competition, see \cite{cournot1838recherches},
the producers are controling their production $q$.
Like in the formulation of the
MFG arising in such a problem
in \cite{MR3359708},
here we consider the Bertrand's formulation \cite{Bertrand},
where an agent directly controls her selling price
$\aa=P(q,\qo)$.
Then after inverting the latter equality,
the production can be viewed as a function
of the price and the mean field.
Mathematically this corresponds to writing
$q=Q\lp \aa,\aao\rp$.

In \cite{MR3359708},
the authors considered a linear demand system depending on
$\qo_{\text{lin}}=\int_{\RR}q(x)dm(x)$,
and a price satisfying
$\aa=P_{\text{lin}}(q,\qo_{\text{lin}})=1-q-\ee \qo_{\text{lin}}$.
In this case, $L^{\text{lin}}$ the running cost,
and $H_{\text{lin}}$ its Legendre transform
are defined by
\begin{align*}
    L^{\text{lin}}\lp \aa,\mu\rp
    &=
    \aa^2
    +\frac{\ee}{1+\ee}\aa\aao
    -\frac1{1+\ee}\aa,
    \\
    H^{\text{lin}}\lp p,\mu\rp
    &=
    \frac14\lp p+\frac{\ee}{1+\ee}\aao-\frac1{1+\ee}\rp^2,
\end{align*}
where $\aa,p\in\RR$, $\mu\in\cP\lp\RR\times\RR\rp$
and $\aao$ is defined by
$\aao=\int_{\RR\times\RR}\aat d\mu(y,\aat)$.
Therefore the system of MFGC has the following form,
\begin{equation}
    \label{eq:exhaulin}
    \lc
    \begin{aligned}
        &-\ptt u
        -\nu\Delta u
        +
        \frac14\lp \nabla_xu+\frac{\ee}{1+\ee}\aao-\frac1{1+\ee}\rp^2
        =
        0,
        \\
        &\ptt m
        -\nu\Delta m
        -\divo\lp\frac12\lp \nabla_xu+\frac{\ee}{1+\ee}\aao-\frac1{1+\ee}\rp m\rp
        =
        0,
        \\
        &\aao(t)
        =
        -\int_{\RR^d}
        \frac12\lp \nabla_xu+\frac{\ee}{1+\ee}\aao(t)-\frac1{1+\ee}\rp
        dm(t,x),
        \\
        &u(T,x)
        =
        g(x),
        \\
        &m(0,x)
        =
        m_0(x),
    \end{aligned}
    \right.
\end{equation}
for $(t,x)\in[0,T]\times\RR$.
Roughly speaking, $\ee=0$ corresponds
to a monopolist who does not
suffer from competition,
and she plays as if she was alone in the game.
Conversely, $\ee=\infty$ stands for all the producers
selling the same ressource and the consumers not having
any a priori preference.

Let us consider the following generalization of the latter system
to the $d$-dimensional case with a more general Hamiltonian
and interaction through controls,
\begin{equation}
    \label{eq:exhau}
    \lc
    \begin{aligned}
        &-\ptt u
        -\nu\Delta u
        +
        H\lp t,x,\nabla_xu+\vp(x)^TP(t)\rp
        =
        f(t,x,m(t)),
        \\
        &\ptt m
        -\nu\Delta m
        -\divo\lp H_p\lp t, x,\nabla_xu+\vp(x)^TP(t)\rp m\rp
        =
        0,
        \\
        &P(t)
        =
        \Psi\lp t,
        -\int_{\RR^d}\vp(x)
        H_p\lp t,x,\nabla_xu+\vp(x)^TP(t)\rp dm(t,x)\rp,
        \\
        &u(T,x)
        =
        g(x,m(T)),
        \\
        &m(0,x)
        =
        m_0(x),
    \end{aligned}
    \right.
\end{equation}
where $\vp:\RR^d\mapsto\RR^{d\times d}$
and $\Psi:\RR^d\mapsto\RR^{d\times d}$
are given functions.
The counterpart of the latter system
posed on $\TT^d$ has been
introduced in \cite{2019arXiv190205461F}.
Theorem \ref{thm:UniMonob} and \ref{thm:bpower}
provide the existence and the uniqueness respectively
of the solution of this MFGC system.

\begin{proposition}
    \label{prop:uniexhau}
    Assume
    \ref{hypo:Lreg},
    \ref{hypo:Lbconv},
    \ref{hypo:gMono}.
    If the function $\Psi$ is continuous,
    $\Psi(t,\cdot)$ is monotone,
    locally Lipschitz continuous,
    and admits at most a
    power-like growth
    of exponent $q'-1$
    with a coefficient uniform in $t\in[0,T]$,
    there exists at most one solution
    to \eqref{eq:exhau}.
\end{proposition}

\begin{proposition}
    \label{prop:exiexhau}
        Assume 
        \ref{hypo:Lreg},
        \ref{hypo:Lbconv},
        \ref{hypo:Lcoer}-\ref{hypo:Monogregx},
        and that $\Psi$ satisfies
        the same assumptions as in
        Proposition \ref{prop:uniexhau}.
        %and 
        %\ref{hypo:Laabound},
        There exists a solution to \eqref{eq:exhau}.
        %
        %Moreover if $\vp$ does not depend on $x$,
        %the existence holds
        %without assuming \ref{hypo:Laabound}.
\end{proposition}
\begin{proof}
    Take the drift function as $b=\aa$.
    We define the Lagrangian
    $\ell$ by
    \begin{equation*}
        \ell\lp t,x,\aa,\mu\rp
        =
        L\lp t,x,\aa\rp
        +\vp(x)\aa\cdot P(t,\mu)
        +f(t,x,m),
    \end{equation*}
    where $L$ is the Legendre transform
    of the map $H$ in \eqref{eq:exhau},
    and $P(t,\mu)$ is defined by 
    $P(t,\mu)=\Psi\lp t,
    \int_{\RR^d\times\RR^d}\vp(x)\aa d\mu(x,\aa)\rp$,
    for $(t,x,\aa,m,\mu)\in
    [0,T]\times\RR^d\times\RR^d\times
    \cP\lp\RR^d\rp\times
    \cP\lp\RR^d\times\RR^d\rp$
    such that $m$ is the first marginal
    of $\mu$.
    We take $h$ as the Legendre transform
    of $\ell$ with respect to $\aa$.

    If $\Psi$ satisfies the assumptions in \ref{prop:uniexhau},
    any of the assumptions
    \ref{hypo:Lreg},
    \ref{hypo:Lbconv},
    \ref{hypo:Lcoer},
    or \ref{hypo:Lbound}
    is preserved by replacing $L$ by $\ell$.
    Moreover, a straightforward calculation
    yields that
        \begin{multline*}
            \int_{\RR^d\times\RR^d}
            \lp \ell\lp t, x,\a,\mu^1 \rp
            -\ell\lp t, x,\a,\mu^2\rp \rp
            d\lp \mu^1-\mu^2 \rp (x,\aa)
            \\
            =
            \lp P\lp t,\mu^1\rp
            -P\lp t,\mu^2\rp\rp\cdot
            \int_{\RR^d\times\RR^d}\vp(x)\aa d\lp\mu^1-\mu^2\rp(x,\aa),
        \end{multline*}
    for $t\in[0,T]$ and $\mu^1,\mu^2\in\cP\lp\RR^d\times\RR^d\rp$.
    This and the monotonicity of $\Psi$ implies
    that  $\ell$ satisfies \ref{hypo:LMono}.
    Therefore, Propositions \ref{prop:uniexhau} and \ref{prop:exiexhau}
    are direct consequences of Theorems
    \ref{thm:UniMonob} and \ref{thm:bpower} respectively.
\end{proof}
In \cite{2019arXiv190205461F},
similar existence and uniqueness results
for the counterpart of \eqref{eq:exhau} posed on $\TT^d$
are given in the quadratic setting,
with a uniformly convex Lagrangian
and $\Psi$ being the gradient of a convex map.
Here, we generalize their results to a wider
class of Lagrangians and functions $\Psi$.

For an extension of this model to the case when $\Psi$
is non-monotone, see \cite{2019arXiv190411292K}.

%Such non-monotone functions for $d>1$,
%allow to consider producers selling several ressources
%which may be negatively correlated,
%like cars and oil
%(if the production of cars increases,
%then the demand for oil also increases
%and thus the price of oil rises while the price of cars decreases),
%or pesticides and medicines,
%or gold and goods.
%To the author's knowledge,
%such a generalization of the exhaustible ressource model
%to negatively correlated ressources
%is new in the MFG literature.

\subsection{A model of crowd motion}
\label{subsec:abncrowd}
This model of crowd motion has been
introduced in \cite{2019arXiv190411292K}
in the non-monotone setting.
It has been
numerically studied
in \cite{achdou2020mean}
in the quadratic non-monotone case.
For $\mu\in\cP\lp\RR^d\times\RR^d\rp$
we define $V(\mu)$ the average drift by, 
\begin{equation*}
    V(\mu)
    =
    \frac1{Z(\mu)}
    \int_{\RR^d\times\RR^d}
    \aa k(x) d\mu(x,\aa),
\end{equation*}
where $Z(\mu)$ is a normalization constant
defined by
$Z(\mu)
=
\lp\int_{\RR^d\times\RR^d}
k(x)^{q_1} d\mu(x,\aa)\rp^{\frac1{q_1}}$,
for some constant $q_1\in[q,\infty]$
where $q$ is defined below.
To be consistent with the notations used in \cite{2019arXiv190411292K},
$k:\RR^d\rightarrow\RR_+$ is a non-negative kernel.
By convention, if $Z\lp\mu\rp=0$, we take $V\lp\mu\rp=0$.
%with bounded derivatives.

The state of a representative agent is given by her position $X_t\in\RR^d$
which she controls through her velocity $\aa$ via the following
stochastic differential equation,
\begin{equation*}
    dX_t
    =
    \aa_t dt
    +\sqrt{2\nu}dW_t.
\end{equation*}
Her objective is to minimize the cost functional given by,
\begin{equation*}
    \EE\lb\int_0^T
    \frac{\th}{2}\labs \aa_t+\ll V\lp\mu(t)\rp\rabs^{2}
    +\frac{1-\th}{a'}\labs \aa_t\rabs^{a'}
    +f(t,X_t,m(t))dt
    +g(X_T,m(T))\rb,
\end{equation*}
where $\ll\geq0$ and $0\leq\th\leq1$ are two constants 
standing for the intensity of the preference
of an individual to have an opposite
control as the stream one,
and $a'> 1$ is an exponent.
In this model we define the Lagrangian $L$ by,
\begin{equation*}
    L\lp x,\aa,\mu\rp
    =
    \frac{\th}{2}\labs\aa+\ll V(\mu)\rabs^{2}
    +\frac{1-\th}{a'}\labs\aa\rabs^{a'},
\end{equation*}
and the Hamiltonian $H$ as its Legendre transform.
The map $H$ does not admit an explicit form for every choice
of the parameters $a'$.
We take $q'=\max\lp 2,a'\rp$,
and $q=\frac{q'}{q'-1}$ its conjugate exponent.

Here, since the control is equal to the drift,
the MFGC system is of the form of \eqref{eq:MFGCrd}.
Therefore, the following proposition
is a consequence of 
Theorems
\ref{thm:UniMonob}
and \ref{thm:bpower}.
\begin{proposition}
    Under assumption \ref{hypo:Monogregx},
    there exists a solution to the above
    MFGC system of crowd motion.

    Under assumption \ref{hypo:gMono},
    this solution is unique.
\end{proposition}
The proof is straightforward,
it consists in checking
that $L$ satisfies
\ref{hypo:Lreg}-\ref{hypo:Lbound}.

For existence results of the MFGC system
of this model with $\ll<0$,
see \cite{2019arXiv190411292K}.

\section{The fixed point \eqref{eq:muttda}
    and the proof of Lemma \ref{lem:FPmucont}}
\label{sec:MonoFPV}
%The results and the proofs
%are given for \eqref{eq:murd}.
%However, everything holds when
%the state space is the torus,
%and the proofs are easier in this case.
This section is devoted to step \ref{step:FPmu}.
In paragraph \ref{subsec:FPmuMono},
we state a priori estimates
on a fixed point of \eqref{eq:muttda}
(Lemma \ref{lem:apriorimu});
then we we use these estimates and
Leray-Schauder fixed point theorem
(Theorem \ref{thm:Leray})
and obtain the existence of a fixed point
\eqref{eq:muttda} at any time $t\in[0,T]$
(Lemma \ref{lem:FPmuLS}).
We address
the continuity with respect to time of the fixed point,
i.e. step \ref{step:FPmucont},
in Lemma \ref{lem:contmu}.

In this section and the next one,
we work on
$\TT^d_a=\RR^d/\lp a\ZZ^d\rp$
$\TT^d_a$
the $d$-dimensional torus of radius $a>0$.
Here we take
$L:[0,T]\times\TT^d_a\times\RR^d\times\cP\lp\TT^d_a\times\RR^d\rp\to \RR$.
All the assumptions in paragraph \ref{subsec:hypoMono}
are stated in $\RR^d$, but,
when considering that $L$
satisfies one of those assumptions,
we shall simply replace $\RR^d$
by $\TT^d_a$ as the state set in the chosen assumption
(note that we keep $\RR^d$ as the set of admissible controls).
The initial distribution $m_0$
is now in $\cP\lp\TT^d_a\rp$.
The Hamiltonian $H$ is still defined as the
Legendre transform of $L$,
i.e. it satisfies \eqref{eq:defHrd}.

\subsection{Leray-Schauder Theorem for solving
    the fixed point in $\mu$}
\label{subsec:FPmuMono}
We start by stating a priori estimates
for solutions of the fixed point in $\mu$ \eqref{eq:muttda},
involving $\LL_{q'}\lp\mu\rp$ and $\LL_{\infty}\lp\mu\rp$
defined in \eqref{eq:defL}.
\begin{lemma}
    \label{lem:apriorimu}
    Assume that $L$
    satisfies
    \ref{hypo:Lreg}-\ref{hypo:Lbound}
    For any $t\in[0,T]$,
    $m\in \cP\lp\TT^d_a\rp$
    and $p\in C^0\lp\TT^d_a;\RR^d\rp$,
    if there exists
    $\mu\in\cP\lp\TT^d_a\times\RR^d\rp$
    such that
    \begin{equation}
        \label{eq:MonoFPmu}
        \mu
        =
        \lp I_d,-H_p\lp t, \cdot,p(\cdot),\mu\rp\rp\# m,
    \end{equation}
    then it satisfies
    \begin{align}
        \label{eq:aprioriLq'}
            \LL_{q'}\lp\mu\rp^{q'}
            &\leq
            4C_0^2
            +\frac{\lp q'\rp^{q-1}\lp2C_0\rp^q}{q}
            \norm[q]{p}{L^q(m)},
            \\
        \label{eq:aprioriLinf}
            \LL_{\infty}\lp\mu\rp
            &\leq
            C_0\lp 1+\norminf{p}+\LL_{q'}\lp\mu\rp\rp.
    \end{align}
\end{lemma}
\begin{proof}
    We use \ref{hypo:LMono}
    with $m\otimes\delta_0$
    and $\mu$ satisfying \eqref{eq:MonoFPmu},
    \begin{equation}
        \label{eq:LMono_apriori}
        \int_{\TT^d_a\times\RR^d}
        \lp L\lp t,x,\aa, \mu \rp
        -L\lp t,x,\aa, m\otimes\delta_0 \rp\rp
        d\mu(x,\aa)
        +
        \int_{\TT^d_a}
        \lp L\lp t,x,0, m\otimes\delta_0 \rp
        -L\lp t,x,0, \mu\rp\rp
        dm(x)
        \geq
        0.
    \end{equation}
    From \ref{hypo:Lbound},
    we obtain
    \begin{equation}
        \label{eq:Lbounded}
        \int_{\TT^d_a}
        L\lp t,x,0, m\otimes\delta_0 \rp
        dm(x)
        \leq
        C_0. 
    \end{equation}
    The latter two inequalities,
    \ref{hypo:Lcoer}
    and the convexity of $L$
    (stated in Lemma \ref{lem:Lstrconv})
    yield
    \begin{align*}
        C_0^{-1}
        \int_{\TT^d_a\times\RR^d}
        \labs\aa\rabs^{q'}d\mu(x,\aa)
        -C_0
        &\leq
        C_0+
        \int_{\TT^d_a}
        \lp L\lp t,x,\aa^{\mu}(x), \mu \rp
        -L\lp t,x,0, \mu \rp\rp dm(x)
        \\
        &\leq
        C_0+
        \int_{\TT^d_a\times\RR^d}
        \aa\cdot L_{\aa}\lp t,x,\aa, \mu \rp
        d\mu(x,\aa),
    \end{align*}
    where $\aa^{\mu}$ is defined
    in paragraph \ref{subsec:notaMono}.
    We recall that
    $p(x)=-L_{\aa}\lp t,x,\aa^{\mu}(x),\mu\rp$.
    Using the inequality
    $yz\leq
    \frac{y^{q'}}{c^{q'}q'}
    +\frac{c^qz^{q}}{q}$
    which holds for any $y,z,c>0$,
    we obtain
    \begin{equation*}
        \frac1{C_0}\int_{\TT^d_a\times\RR^d}
        \labs\aa\rabs^{q'}d\mu(x,\aa)
        \leq
        2C_0
        +\frac{\lp 2C_0q'\rp^{\frac{q}{q'}}}q
        \int_{\TT^d_a}\labs p(x)\rabs^q dm(x)
        +\frac1{2C_0}
        \int_{\TT^d_a\times\RR^d}
        \labs\aa\rabs^{q'}d\mu(x,\aa).
    \end{equation*}
    This and $\frac{q}{q'}+1=q$
    imply \eqref{eq:aprioriLq'}.
    This and \ref{eq:Hpbound} implies 
    \ref{eq:aprioriLinf},
    we recall that we assume $C_0=\Ct_0$.    
\end{proof}

Here, we shall use
Leray-Schauder fixed point theorem as
stated in \cite{MR1814364} Theorem $11.6$.
\begin{theorem}[Leray-Schauder fixed point theorem]
    \label{thm:Leray}
    Let $\cB$ be a Banach space and let
    $\Psi$ be a compact mapping from
    $[0,1]\times\cB$ into $\cB$
    such that $\Psi(0,x)=0$ for all $x\in\cB$.
    Suppose that there exists a constant $C$ such that
    \begin{equation*}
        \norm{x}{\cB}\leq C,
    \end{equation*}
    for all $\lp \th,x\rp\in[0,1]\times\cB$
    satisfying $x=\Psi(\th,x)$.
    Then the mapping $\Psi(1,\cdot)$
    of $\cB$ into itself
    %given by $\Psi_1x=\Psi(1,x)$
    has a fixed point.
\end{theorem}

From Lemma \ref{lem:apriorimu}
and Theorem \ref{thm:Leray},
we obtain the following existence
result for a fixed point \eqref{eq:muttda}.

\begin{lemma}
    \label{lem:FPmuLS}
    Assume
    \ref{hypo:Lreg}-\ref{hypo:Lbound}.
    %\ref{hypo:Lstrconv}-\ref{hypo:Lbound},
    %and
    %\ref{hypo:Luniconv}.
    For $t\in[0,T]$,
    $m\in\cP\lp\TT^d_a\rp$
    and 
    $p\in C^0\lp\TT^d_a;\RR^d\rp$,
    there exists a unique
    $\mu\in\cP\lp\TT^d_a\times\RR^d\rp$
    such that
        $\mu
        =
        \lp I_d,
        -H_p\lp t,\cdot,p(\cdot),\mu\rp\rp\# m$.
    Moreover, $\mu$ satisfies the inequality
    stated in Lemma \ref{lem:apriorimu}.
\end{lemma}

In the following proof,
we will take advantage of the flexibily
offered when making all assumptions
on the Lagrangian,
instead of the Hamiltonian.
We will introduce a sequence
of new Lagrangians.
The associated Hamiltonians
may not admit explicit form;
therefore it would be difficult
to check assumptions on them.
Here on the one hand,
checking the assumptions
on the new Lagrangians is straightforward.
On the other hand,
we obtain the same conclusions on the
new Hamiltonian as stated in Lemma \ref{lem:Hbound}.

\begin{proof}
    Take $(t,\po,m)$
    satisfying the same assumptions as
    $(t,p,m)$ in
    Lemma \ref{lem:FPmuLS}.
    In order to use the Leray-Schauder
    fixed point theorem later,
    we introduce the following family of Lagrangians
    indexed by $\ll\in[0,1]$,
    \begin{equation}
        \label{eq:Lpo}
        L^{\po,\ll}\lp x,\aa,\mu\rp
        =
        \ll
        L\lp t,x,\aa,\mu\rp
        +(1-\ll)
        \lp\frac{\labs\aa\rabs^{q'}}{q'}
        -\aa\cdot \po(x)\rp,
    \end{equation}
    for $\lp x,\aa,\mu\rp\in\TT^d_a\times\RR^d\times\cP\lp\TT^d_a\times\RR^d\rp$.
    We denote by $H^{\po,\ll}$ the Legendre
    transform of $L^{\po,\ll}$.
    For $\ll=0$ it satisfies
    \begin{equation}
        \label{eq:Hpo0}
        H^{\po,0}\lp x,p,\mu\rp
        =
        \frac1q\labs p-\po(x)\rabs^q.
    \end{equation}
    From Young inequality,
    we obtain that
    \begin{align*}
        \labs\aa\cdot\po(x)\rabs
        \leq
        \frac{\labs\aa\rabs^{q'}}{2q'}
        +
        \frac{2^{q-1}}{q}\norminf{p}.
    \end{align*}
    Therefore,
    up to changing $C_0$
    into $\max\lp \frac1{2q'},\frac{2^{q-1}}{q}\norminf{p},C_0\rp$,
    we may assume that
    $L^{\po,\ll}$
    satisfies 
    \ref{hypo:Lreg}-\ref{hypo:Lbound},
    with the same constant $C_0$
    for any $\ll\in[0,1]$.
    The map
    $\lp\ll,x,p,\mu\rp\mapsto -H^{\po,\ll}_p\lp x,p,\mu\rp$
    is continuous on 
    $[0,1]\times\TT^d_a\times\RR^d\times\cP_{\infty,R}\lp\TT^d_a\times\RR^d\rp$,
    for any $R>0$,
    by the same arguments as in the proof
    of Lemma \ref{lem:Hbound}.

    For $\aa\in C^0\lp\TT^d_a;\RR^d\rp$,
    we set 
    $\mu=\lp I_d,\aa\rp\#m\in\cP\lp\TT^d_a\times\RR^d\rp$
    and $\aao(x)=-H_p^{\po,\ll}\lp x,\po(x),\mu\rp$,
    for $x\in\TT^d_a$.
    We define the map $\Psi$,
    from $[0,1]\times C^0\lp\TT^d_a;\RR^d\rp$
    to $C^0\lp\TT^d_a;\RR^d\rp$,
    by $\Psi\lp\ll,\aa\rp=\aao$.
    If $\aa$ is a fixed point of $\Psi(1,\cdot)$,
    then $\mu$ defined as above 
    satisfies the fixed point in
    Lemma \ref{lem:FPmuLS}.
    Conversely, if $\mu$ satisfies the
    fixed point in 
    Lemma \ref{lem:FPmuLS},
    then $\aa^{\mu}$ (defined in paragraph
    \ref{subsec:notaMono})
    is a fixed point of $\Psi(1,\cdot)$.

    The map $\Psi$ is continuous by the continuity
    of $\lp\ll,x,p,\mu\rp\mapsto -H^{\po,\ll}_p\lp x,p,\mu\rp$.
    For $R>0$, the set $A_R$, defined by
    $A_R=[0,1]\times\TT^d_a\times B_{\RR^d}\lp 0,R\rp
    \times\cP_{\infty,R}\lp\TT^d_a\times\RR^d\rp$,
    is compact.
    By Heine theorem,
    the map
    $\lp\ll,x,p,\mu\rp\mapsto -H^{\po,\ll}_p\lp x,p,\mu\rp$
    is uniformly continuous on $A_R$.
    Here, note that we use the fact that
    $\cP_{\infty,R}\lp\TT^d_a\times\RR^d\rp$
    is a metric space since 
    the weak* topology coincides with the
    topology induced by the
    $1$-Wassertein distance
    on $\cP_{\infty,R}\lp\TT^d_a\times\RR^d\rp$.
    Heine theorem also implies that $\po$
    is uniformly continuous.
    Therefore, $\Psi$ is a compact mapping from 
    $[0,1]\times C^0\lp\TT^d_a;\RR^d\rp$
    to $C^0\lp\TT^d_a;\RR^d\rp$, i.e.
    it maps bounded subsets of
    $[0,1]\times C^0\lp\TT^d_a;\RR^d\rp$
    into relatively compact subsets of
    $C^0\lp\TT^d_a;\RR^d\rp$:
    this comes from the latter observation
    and Arzelà-Ascoli theorem.

    Take $\aa$ a fixed point
    of $\Psi(\ll,\cdot)$,
    for $\ll\in[0,1]$,
    Lemma \ref{lem:apriorimu}
    implies that $\norminf{\aa}$
    is bounded by a constant $C$
    which does not depend on $\ll$.
    
    Moreover, it is straightforward to
    check that $\Psi(0,\cdot)=0$.
    Leray-Schauder Theorem \ref{thm:Leray}
    implies that there exists a fixed point
    of the map $\aa\mapsto\Psi\lp1,\aa\rp$,
    which concludes the existence part of the proof.

    The proof of uniqueness relies on
    \ref{hypo:LMono} and the strict convexity of $L$,
    see \cite{MR3805247} Lemma $5.2$ for the detailed proof.
\end{proof}

\subsection{The continuity of the fixed point in time}
\label{subsec:FPmugen}
The fixed point result stated in Lemma \ref{lem:FPmuLS} yields
the existence of a map $(t,p,m)\mapsto \mu$.
The continuity of this map is addressed in the following lemma:
\begin{lemma}
    \label{lem:contmu}
    Assume
    \ref{hypo:Lreg}-\ref{hypo:Lbound}.
    Let
    $\lp t^n,m^n,p^n\rp_{n\in\NN}$
    be a sequence in
    $[0,T]\times
    \cP\lp\TT^d_a\rp\times
    C^0\lp\TT^d_a;\RR^d\rp$.
    Assume that
    \begin{itemize}
        \item
            $t^n\to_{n\to\infty}t\in[0,T]$,
        \item
            $\lp p^n\rp_{n\in\NN}$ is uniformly convergent
            to $p\in C^0\lp\TT^d_a;\RR^d\rp$,
        \item
            $\lp m^n\rp_{n\in\NN}$
            tends to $m$ in the weak* topology.
    \end{itemize}
    We define $\mu^n$ and $\mu$ as the unique solutions
    of the fixed point relation of Lemma
    \ref{lem:FPmuLS}
    respectively associated to
    $\lp t^n,m^n,p^n\rp$
    and 
    $\lp t,m,p\rp$,
    for $n\in\NN$.
    Then the sequence $\lp\mu^n\rp_{n\in\NN}$
    tends to $\mu$
    in $\cP\lp\TT^d_a\times\RR^d\rp$
    equipped with the weak* topology.
\end{lemma}

\begin{proof}
    %Take $\mu$ and $\lp\mu^n\rp_{n\in\NN}$
    %as in Lemma \ref{lem:contmu}.
    The sequence $\lp p^{n}\rp_{n\in\NN}$
    it is uniformly bounded in the norm
    $\norminf{\cdot}$.
    Therefore $\lp\mu^n\rp_{n\in\NN}$
    is uniformly compactly supported
    by Lemma \ref{lem:apriorimu}.
    Thus we can extract a subsequence
    $\lp\mu^{\vp(n)}\rp_{n\in\NN}$
     convergent to some limit
    $\mut\in\cP\lp\RR^d\times\RR^d\rp$
    in the weak* toplogy on measures.

    We recall that $\aa^{\mu}$
    is defined in paragraph
    \ref{subsec:notaMono}.
    Here, since $\mu^{\vp(n)}$
    and $\mu$ are fixed points
    like in Lemma \ref{lem:FPmuLS},
    they satisfy:
    \begin{align*}
        \aa^{\mu^{\vp(n)}}(x)
        &=
        -H_p\lp t^{\vp(n)},x,p^{\vp(n)}(x),\mu^{\vp(n)}\rp,
        \\
        \aa^{\mu}(x)
        &=
        -H_p\lp t,x,p(x),\mu\rp,
    \end{align*}
    for $x\in\TT^d_a$ and $n\in\NN$.
    From the continuity of $H_p$
    stated in
    Lemma \ref{lem:Hbound},
    $\lp\aa^{\vp(n)}\rp_{n\in\NN}$
    tends uniformly to the function
    $\aat:x\mapsto-H_p\lp t,x,p,\mut\rp$.
    Then
    $\lp\lp I_d,\aa^{\vp(n)}\rp\#m^n\rp_{n\in\NN}$
    tends to 
    $\lp I_d,\aat\rp\#m$
    in the weak* topology.
    Hence $\mut$ satisfies the same
    fixed point relation as $\mu$;
    by uniqueness we deduce that
    $\mut=\mu$.
    This implies that all the convergent subsequences
    of $\lp\mu^{n}\rp_{n\in\NN}$
    have the same limit $\mu$,
    thus the whole sequence converges to $\mu$.
\end{proof}

Lemma \ref{lem:FPmuLS} states that
for all time the fixed point
\eqref{eq:muttda} has a unique solution.
Then Lemma \ref{lem:contmu}
yields the continuity of the map defined
by the fixed point under suitable assumptions.
Therefore, the conclusion of step
\ref{step:FPmucont}
and the Lemma \ref{lem:FPmucont}
are straightforward consequences of
these two lemmas.

\begin{remark}
    All the conclusions of this section
    hold when we relax Assumption \ref{hypo:LMono},
    assuming that the inequality holds only
    when $\mu^1$ and $\mu^2$
    have the same first marginal.
    Some applications of MFGC do not satisfy
    \ref{hypo:LMono}, but satisfy the above-mentioned
    relaxed monotonicity assumption.
    This is the case of the MFG version of the
    Almgren and Chriss' model for price impact and
    high-frenquency trading, discussed in
    \cite{MR3805247,MR3752669,MR3325272,2019arXiv190411292K}.
    
    However, the a priori estimates in the next section
    do not hold under this relaxed monotonicity assumption.
    We refer to \cite{2019arXiv190411292K} for estimates
    which do not rely on \ref{hypo:LMono}
    (Assumptions {\bf FP1} and {\bf FP2}
    in \cite{2019arXiv190411292K} are unnecessary
    if $L$ satisfies the relaxed monotonicity assumption).
\end{remark}

\section{A priori estimates for the
solutions to \eqref{eq:MFGCttda}}
\label{sec:aprioriMono}
%%%%%%%%%%%%%%%%
In order to use the Leray-Schauder
fixed point theorem later,
we introduce the following family of Lagrangians
indexed by $\th\in(0,1]$,
\begin{equation}
    \label{eq:Lth}
    L^{\th}\lp t,x,\aa,\mu\rp
    =
    \th
    L\lp t,x,\th^{-1}\aa,\Th(\mu)\rp,
\end{equation}
where the map $\Th:\cP\lp\TT^d_a\times\RR^d\rp
\rightarrow\cP\lp\TT^d_a\times\RR^d\rp$
is defined by
$\Th(\mu)=\lp I_d\otimes\th^{-1}I_d\rp\#\mu$.
%if $\th>0$,
%and $\Th(\mu)=\lp I_d\otimes0\rp\#\mu$
%if $\th=0$.
Then the Hamiltonian defined as the Legendre
transform of $L^{\th}$ is given by
\begin{equation}
    \label{eq:Hth}
    H^{\th}\lp t,x,p,\mu\rp
    =
    \th
    H\lp t,x,p,\Th(\mu)\rp.
\end{equation}
The definition of the Hamiltonian can naturally be extended
to $\th=0$ by $H^0=0$,
the associated Lagrangian is $L^0=0$ if $\aa=0$
and $L^0=\infty$ otherwise.
We introduce the following system of MFGC,
\begin{subequations}\label{eq:MFGCth}
     \begin{empheq}{align}
    \label{eq:HJBth}
    &-\ptt u(t,x)
    - \nu\Delta u(t,x)
    + H^{\th}\lp t,x, \nabla_xu(t,x),\mu(t)\rp
    =
    \th f(t,x,m(t))
    &\text{ in }
    (0,T)\times\TT^d_a,\\
    \label{eq:FPKth}
    & \ptt m(t,x)
    - \nu\Delta m(t,x)
    -\divo\lp H^{\th}_p\lp t,x,\nabla_xu(t,x),\mu(t)\rp m\rp
    =
    0
    &\text{ in }
    (0,T)\times\TT^d_a,\\
    \label{eq:muth}
    &\mu(t)
    = \Bigl(
    I_d,
    - H^{\th}_p\lp t,\cdot,\nabla_xu(t,\cdot),\mu(t)\rp
    \Bigr){\#}m(t)
    &\text{ in } [0,T],\\
    \label{eq:CFuth}
    &u(T,x)
    =
    \th g(x,m(T))
    &\text{ in }
    \TT^d_a,\\
    \label{eq:CImth}
    &m(0,x)
    =
    m_0(x)
    &\text{ in }
    \TT^d_a.
    \end{empheq}
\end{subequations}
When $\th=1$,
the latter system coincides with
\eqref{eq:MFGCttda}.
When $\th=0$, \eqref{eq:MFGCth} consists in a situation
in which the state of a representative agent
satisfies a non-controlled
stochastic differential equation.
Alternatively it can be interpreted as
a game in which the agents pay an infinite
price as soon as they try to use a control different than $0$.
%Therefore in this case, there is no interaction between agents.
In particular the case $\th=0$ is specific and easier than the case
when $\th>0$.
Therefore, in the rest of this section, we only consider
$\th\in(0,1]$.

Let us mention that assumptions
\ref{hypo:Lreg}-\ref{hypo:LMono}
%\ref{hypo:Lstrconv}-\ref{hypo:LMono}
are preserved when
replacing $L$ and $H$ by
$L^{\th}$ and $H^{\th}$ respectively.
Moreover the
inequalities from
\ref{hypo:Lcoer},
\ref{hypo:Lbound},
become respectively
\begin{align}
    \label{eq:Lthcoer}
        L^{\th}(t,x,\aa,\mu)
        &\geq
        C_0^{-1}\th^{1-q'}|\aa|^{q'}
        -C_0\th
        -C_0\th^{1-q'}\LL_{q'}\lp\mu\rp^{q'},
        \\
    \label{eq:Lthbound}
    \labs L^{\th}(t,x,\aa,\mu)\rabs
        &\leq
        C_0\th
        +C_0\th^{1-q'}\lp|\aa|^{q'}
        +\LL_{q'}\lp\mu\rp^{q'}\rp,
\end{align}
since $\LL_{q'}\lp \Th(\mu)\rp=\th^{-1}\LL_{q'}\lp\mu\rp$.
Furthermore, the conclusions
of Lemma \ref{lem:Hbound} hold
and the inequalities become respectively
\begin{align}
    \label{eq:Hthpbound}
        \labs H^{\th}_p\lp t,x,p,\mu\rp\rabs
        &\leq
        C_0\th\lp 1+|p|^{q-1}\rp
        +C_0\LL_{q'}\lp\mu\rp,
        \\
    \label{eq:Hthbound}
        \labs H^{\th}\lp t,x,p,\mu\rp\rabs
        &\leq
        C_0\th\lp 1+|p|^{q}\rp
        +C_0\th^{1-q'}\LL_{q'}\lp\mu\rp^{q'},
        \\
    \label{eq:Hthcoer}
        p\cdot H^{\th}_p\lp t,x,p,\mu\rp
        -H^{\th}\lp t,x,p,\mu\rp
        &\geq
        C_0^{-1}\th|p|^q
        -C_0\th
        -C_0\th^{1-q'}\LL_{q'}\lp\mu\rp^{q'},
        \\
    \label{eq:Hthxbound}
        \labs H^{\th}_x\lp t,x,p,\mu\rp\rabs
        &\leq
        C_0\th\lp 1+\labs p\rabs^{q}\rp
        +C_0\th^{1-q'}\LL_{q'}\lp\mu\rp^{q'}.
\end{align}
We recall that without loss of generality,
we assumed $\Ct_0=C_0$ where $\Ct_0$
is defined in Lemma \ref{lem:Hbound}.

Instead of proving Lemma \ref{lem:Monoapriorith1}
and step \ref{step:ttdaapriori},
we address the more general following lemma
which provides a priori estimates
not only for solutions to \eqref{eq:MFGCttda}
but also for solutions to \eqref{eq:MFGCth}.
This will help to use the Leray-Schauder theorem
in the next section.
\begin{lemma}
    \label{lem:Monoapriori}
    Under assumptions
    \ref{hypo:Lreg}-\ref{hypo:Monogregx},
    %\ref{hypo:Lstrconv}-\ref{hypo:Monogregx},
    there exists a positive constant $C$
    which only depends on the constants in the assumptions
    and not on $a$ or $\th$,
    such that the solution to \eqref{eq:MFGCth}
    satisfies:
        %$\int_0^TL_{q'}\lp\mu(t)\rp^{q'} dt
        %\leq
        %C\th^{q'}$,
        $\norminf{u}
        \leq
        C\th$,
        $\norminf{\nabla_xu}
        \leq
        C\th^{\frac12}$ and
        $\displaystyle{\sup_{t\in[0,T]}} \LL_{\infty}\lp\mu(t)\rp
        \leq
        C\th$.
\end{lemma}

\begin{proof}
    \emph{First step: controlling
    $\int_0^T\LL_{q'}\lp\mu(t)\rp^{q'} dt$}

    Let us take $(X,\aa)$ defined by
    \begin{equation*}
        \lc
        \begin{aligned}
            \aa_t
            &=
            \aa^{\mu(t)}(t,X_t)
            =
            -H_p^{\th}\lp t, X_t,\nabla_xu(t,X_t),\mu(t)\rp,
            \\
            dX_t
            &=
            \aa_tdt
            +\sqrt{2\nu}dB_t,
            \\
            X_0
            &=
            \xi\sim m_0,
        \end{aligned}
        \right.
    \end{equation*}
    where $\lp B_t\rp_{t\in[0,T]}$
    is a Brownian motion independent
    of $\xi$.

    The function $u$ is
    the value function of an
    optimization problem,
    i.e. the lowest
    cost that a representative agent can
    achieve from time $t$ to $T$
    if $X_t=x$,
    when the probability measures
    $m$ and $\mu$ are fixed,
    %(a priori $m$ may differ from the law of $X$,
    %and $\mu$ may differ from the joint law of $(X,\aa)$),
    i.e.
    \begin{equation}
        \label{eq:HJBopt}
        \aa_{|s\in[t,T]}
        =
        \argmin_{\aa'}
        \EE\lb \int_t^T
        L^{\th}\lp s, X^{\aa'}_s,\a'_s, \mu(s) \rp 
        +\th f\lp s,X^{\aa'}_s,m(s)\rp ds 
        +\th g\lp X^{\aa'}_T,m(T)\rp \rb,
    \end{equation}
    where for a control $\aa'$,
    we define
    \begin{equation*}
        \lc
        \begin{aligned}
            dX^{\aa'}_t
            &=
            \aa'_tdt
            +\sqrt{2\nu}dB'_t, \\
            X^{\aa'}_0
            &=
            \xi'\sim m_0,
        \end{aligned}
        \right.
    \end{equation*}
    and $\lp B'_t\rp_{t\in[0,T]}$
    is a Brownian motion independent
    of $\xi'$.
    Let us recall that
    for any $t\in[0,T]$,
    $m(t)$ is the law of $X_t$,
    and $\mu(t)$ is the law of
    $\lp X_t,\aa_t\rp$.
    We introduce $\Xt$
    the stochastic process
    defined by
    \begin{equation*}
        \lc
        \begin{aligned}
            d\Xt_t
            &=
            \sqrt{2\nu}dB_t, \\
            \Xt_0
            &=
            \xi\sim m_0.
        \end{aligned}
        \right.
    \end{equation*}
    We set
    $\mt(t)=\cL(\Xt_t)$ and
    $\mut(t)=\cL(\Xt_t)\otimes\dd_0$
    for $t\in[0,T]$.
    For the strategy consisting in taking $\alpha'=0$,
    \eqref{eq:HJBopt} yields the inequality:
    \begin{multline*}
        \int_0^T 
        \int_{\TT^d_a\times\RR^d}
        L^{\th}\lp t,x,\aa, \mu(t) \rp   d\mu(t,x,\aa) dt
        +\int_0^T\int_{\TT^d_a}\th f\lp t,x,m(t)\rp dm(t,x)dt
        +\int_{\TT^d_a}\th g\lp x,m(T)\rp dm(T,x)
        \\
        \leq
        \int_0^T
        \int_{\TT^d_a\times\RR^d}
        L^{\th}\lp t,x,\aa, \mu(t) \rp d\mut(t,x,\aa)dt
        +\int_0^T\int_{\TT^d_a}\th f\lp t,x,m(t)\rp d\mt(t,x)dt
        +\int_{\TT^d_a}\th g\lp x,m(T)\rp d\mt(T,x). 
    \end{multline*}
    This and
    \ref{hypo:Monogregx} imply that,
    \begin{equation}
        \label{eq:apriori_mono1}
        \int_0^T
        \int_{\TT^d_a\times\RR^d}
        L^{\th}\lp t,x,\aa, \mu(t)\rp d\mu(t,x,\aa) dt
        \leq
        \int_0^T
        \int_{\TT^d_a\times\RR^d}
        L^{\th}\lp t,x,\aa, \mu(t)\rp d\mut(t,x,\aa)dt
        +2C_0\th\lp1+T\rp.
    \end{equation}
    Assumption \ref{hypo:LMono}
    with $(\mu(t),\mut(t))$
    yields
    \begin{multline}
        \label{eq:apriori_mono2}
        \int_{\TT^d_a\times\RR^d}
        L^{\th}\lp t,x,\aa, \mu(t) \rp d\mut(t,x,\aa)+
        \int_{\TT^d_a\times\RR^d}
        L^{\th}\lp t,x,\aa, \mut(t) \rp d\mu(t,x,\aa) \\
        \leq
        \int_{\TT^d_a\times\RR^d}
        L^{\th}\lp t,x,\aa, \mu(t) \rp d\mu(t,x,\aa)+
        \int_{\TT^d_a\times\RR^d}
        L^{\th}\lp t,x,\aa, \mut(t)\rp d\mut(t,x,\aa). 
    \end{multline}
    Moreover, from \ref{hypo:Lbound} we obtain that,
    \begin{equation}
        \label{eq:apriori_mono3}
        \int_{\TT^d_a\times\RR^d}
        L^{\th}\lp t,x,\a, \mut(t) \rp   \mut(t,d(x,\a))
        = 
        \int_{\TT^d_a}
        \th L\lp t,x,0, \mut(t) \rp   \mt(t,dx)
        \leq
        C_0\th. 
    \end{equation}
    Therefore from
    \eqref{hypo:Lcoer},
    \eqref{eq:apriori_mono1},
    \eqref{eq:apriori_mono2} and
    \eqref{eq:apriori_mono3},
    we obtain,
    \begin{align*}
        \int_0^T
        \int_{\TT^d_a}
        \lp C_0^{-1}
        \th^{1-q'}\labs\aa\rabs^{q'}
        -C_0\th\rp d\mu(t,x,\aa)dt
        &\leq
        \int_0^T
        \int_{\TT^d_a\times\RR^d}
        L^{\th}\lp t,x,\aa, \mut(t) \rp d\mu(t,x,\aa)dt
        \\
        &\leq
        C_0\th(2+3T).
    \end{align*}
    This implies
    \begin{equation}
        \label{eq:apriori_intmu}
        \int_0^T
        \LL_{q'}\lp\mu(t)\rp^{q'} dt
        \leq 
        2C_0^{2}
        \th^{q'}
        (1+2T).
    \end{equation}

    \emph{Second step: the uniform estimate on $\norminf{u}$}

    Let us rewrite \eqref{eq:HJBth} in the following way,
    \begin{equation*}
        -\ptt u
        - \nu\Delta u
        + \lb \int_0^1 H^{\th}_p(t,x,s\nabla_x u,\mu(t))ds\rb
        \cdot \nabla_x u
        =
        H^{\th}(t,x, 0,\mu(t))
        +\th f(t,x,m(t)),
    \end{equation*}
    for $(t,x)\in(0,T)\times\TT^d_a$.
    The maximum principle for second-order parabolic equation,
    \ref{hypo:Monogregx},
    and \eqref{eq:Hbound}
    yield that
    \begin{equation*}
        \norminf{u}
        \leq
        C_0\th(1+2T)
        +C_0\th^{1-q'}\int_0^T
        \LL_{q'}\lp\mu(t)\rp^{q'} dt,
    \end{equation*}
    which implies that $u$ is uniformly bounded
    using the conclusion of the previous step.

    \emph{Third step: the uniform estimate on $\norminf{\nabla_xu}$.}
    
    The proof of this step relies on the same Bernstein-like method
    introduced in \cite{2019arXiv190411292K} Lemma $6.5$.
    We refer to the proof of the latter results for more details
    in the derivation of the equations below.

    Let us introduce
    $\rho\in C^{\infty}\lp\lb-\frac{a}2,\frac{a}2\rp^d \rp$
    a nonnegative mollifier
    such that $\rho(x)=0$ if $|x|\geq \frac{a}4$ 
    and $\int_{\lb-\frac{a}2,\frac{a}2\rp^d}\rho(x)dx=1$.
    %By an abuse of notation,
    %we also denote $\rho\in C^{\infty}\lp\TT^d_a\rp$
    %the composition of $\rho$ with the canonical
    %injection from $\TT^d_a$
    %to $[-a,a)^d$.
    For any $0<\dd<1$ and $t\in [0,T]$,
    we introduce
    $\rho^{\dd}=\dd^{-d}\rho\lp \frac{\cdot}{\dd}\rp$
    and 
    $u^{\dd}(t)=\rho^{\dd}\star u(t)$
    with $\star$ being the convolution operator
    with respect to the state variable.

    Possibly after modifying the constant $C$
    appearing in the first step,
    we can assume that 
    $\norminf{u}+\lp1+C_0\rp\th^{1-q'}\int_{0}^T\LL_{q'}\lp\mu(s)\rp^{q'}ds\leq C$
    using the first two steps
    in such a way that $C$ depends only on the constants in the assumptions,
    and not on $\th$.
    Then we introduce $\vp:[-C,C]\rightarrow\RR_+^*$
    and $w^{\dd}$ defined by
    \begin{equation}
    \label{eq:defphiw}
        \begin{aligned}
            &\vp(v)
            =
            \exp\lp\exp\lp
            - v\rp\rp,
            \\ 
            &w^{\dd}(t,x)
            =
            \vp\lp u^{\dd}(T-t,x)
            +\lp1+C_0\rp\th^{1-q'}\int_{T-t}^T\LL_{q'}\lp\mu(s)\rp^{q'}ds\rp
            \labs \nabla_x u^{\dd}\rabs^2(T-t,x),
        \end{aligned}
    \end{equation}
    for
    $(t,x)\in[0,T]\times\TT^d_a$,
    $v\in B_{\RR^d}\lp 0, C\rp$.
    In particular $\vp'<0$,
    and $\vp$, $1/\vp$, $-\vp'$ and $-1/\vp'$
    are uniformly bounded.
    We refer to the proof of Lemma $6.5$
    in \cite{2019arXiv190411292K}
    for the derivation of the following partial
    differential equation satisfied by $w^{\dd}$,
    \begin{multline}
        \label{eq:PDEw}
        \ptt w^{\dd} -\nu\Delta w^{\dd}
        + \nabla_xw^{\dd}\cdot  H^{\th}_p\lp x,\nabla_x u^{\dd}, \mu\rp 
        -2\nu\frac{\vp'}{\vp} \nabla_xw^{\dd}\cdot \nabla_xu^{\dd} 
        +2\nu\vp\labs D^2_{x,x}u^{\dd}\rabs^2
        \\
        =
        \frac{\vp'}{\vp}w^{\dd}\lb
        \nabla_xu^{\dd}\cdot H^{\th}_p\lp x,\nabla_xu^{\dd},\mu\rp
        -H^{\th}\lp x,\nabla_xu^{\dd},\mu\rp
        +\lp1+C_0\rp\th^{1-q'} \LL_{q'}\lp\mu\rp^{q'}\rb
        \\
        -\nu\frac{\vp''\vp-2\lp\vp'\rp^2}{\vp^3}\lp w^{\dd}\rp^2
        -2\vp\nabla_xu^{\dd}\cdot H^{\th}_x\lp x,\nabla_xu^{\dd},\mu\rp
        +2\th\vp\nabla_xu^{\dd}\cdot f_x^{\dd}\lp x,m\rp
        +R^{\dd}(t,x)
    \end{multline}
    in which $H^{\th}$, $f$, $f^{\dd}$, $u$,
    $u^{\dd}$ and $\mu$ are taken at time $T-t$
    and $w^{\dd}$ at time $t$,
    and where $f^{\dd}$ and $R^{\dd}$,
    are defined by,
    \begin{align*}
        f^{\dd}(x,m)
        =&
        \rho^{\dd}\star
        \lp f(\cdot, m)\rp(x),
        \\
        %&R^{\dd}(t,x)
        %=
        %R_1^{\dd}(t,x)
        %+2\vp\nabla_xu^{\dd}\cdot
        %R_2^{\dd}(t,x)
        %\\
        %&\!\begin{multlined}[t][10.5cm]
            R^{\dd}(t,x)
            =&
            -\vp'\labs \nabla_xu^{\dd}\rabs^2
            \lb \rho^{\dd}\star\lp H^{\th}\lp \cdot , \nabla_x u, \mu\rp\rp(x)
            -H^{\th}\lp x,\nabla_xu^{\dd} ,\mu\rp\rb
            \\
            &-2\vp\nabla_xu^{\dd}\cdot\lb\lp \rho^{\dd}\star
            H^{\th}_x\lp\cdot, \nabla_xu(\cdot),\mu\rp\rp(x)
            -H^{\th}_x\lp x,\nabla_xu^{\dd},\mu\rp\rb,
            \\
            &+2\vp\nabla_xu^{\dd}\cdot\lb
            D_{x,x}^2u^{\dd}H^{\th}_p\lp x,\nabla_x u^{\dd},\mu\rp
            -\rho^{\dd}\star \lp D^2_{x,x}uH^{\th}_p\lp\cdot,\nabla_x u,\mu\rp\rp\rb.
        %\end{multlined}
        %\\
        %&R_2^{\dd}(t,x)
        %=
        %D_{x,x}^2u^{\dd}H^{\th}_p\lp x,\nabla_x u^{\dd},\mu\rp
        %-\rho^{\dd}\star \lp D^2_{x,x}uH^{\th}_p\lp\cdot,\nabla_x u,\mu\rp\rp.
    \end{align*}
    From \eqref{eq:Hthcoer},
    we obtain that
    \begin{equation*}
        \nabla_xu^{\dd}\cdot H^{\th}_p\lp x,\nabla_xu^{\dd},\mu\rp
        -H^{\th}\lp x,\nabla_xu^{\dd},\mu\rp
        +\lp1+C_0\rp\th^{1-q'} \LL_{q'}\lp\mu\rp^{q'}
        \geq
        C_0^{-1}\th\labs\nabla_xu^{\dd}\rabs^q
        +\th^{1-q'}\LL_{q'}\lp\mu\rp^{q'}
        -C_0\th.
    \end{equation*}
    Therefore, using \ref{hypo:Monogregx},
    \eqref{eq:PDEw}, \eqref{eq:Hthxbound},
    the facts that $\vp'<0$,
    that $\vp''\vp-2\lp\vp'\rp^2\geq 0$,
    that $\vp$, $\vp^{-1}$,
    $\vp'$, $\lp\vp'\rp^{-1}$ are bounded,
    and the latter inequality,
    we get
    \begin{multline}
        \label{eq:PDIw}
        \ptt w^{\dd} -\nu\Delta w^{\dd}
        + \nabla_xw^{\dd}\cdot  H^{\th}_p\lp x,\nabla_x u^{\dd}, \mu\rp 
        -2\nu\frac{\vp'}{\vp} \nabla_xw^{\dd}\cdot \nabla_xu^{\dd} 
        \\
        \leq
        -C^{-1}\lp \th\lp w^{\dd}\rp^{\frac{q}2}
        +\th^{1-q'}\LL_{q'}\lp\mu\rp^{q'}\rp w^{\dd}
        \\
        +C\lp w^{\dd}\rp^{\frac12}
        \lb \th+\th\lp w^{\dd}\rp^{\frac12}
        +\th\lp w^{\dd}\rp^{\frac{q}2}
        +\th^{1-q'}\LL_{q'}\lp\mu\rp^{q'}\rb
        +\norminf{R^{\dd}},
    \end{multline}
    up to updating $C$.
    We notice that the terms with the highest exponents in $w^{\dd}$
    and $\LL_{q'}\lp\mu\rp^{q'}$ in the right-hand side
    of the latter inequality is non-positive.
    Let us use Young inequalities and obtain
    \begin{align*}
        &\lp w^{\dd}\rp^{\frac12}
        \LL_{q'}\lp\mu\rp^{q'}
        \leq
        \ee w^{\dd}
        \LL_{q'}\lp\mu\rp^{q'}
        +
        \frac1{4\ee}
        \LL_{q'}\lp\mu\rp^{q'},
        \\
        &\lp w^{\dd}\rp^{\qt}
        \leq
        \ee\lp w^{\dd}\rp^{1+\frac{q}2}
        +\frac{q+2-2\qt}{q+2}
        \lp\frac{\ee (q+2)}{2\qt}\rp^{-\frac2{q+2-2\qt}},
    \end{align*}
    for any $\qt<1+\frac{q}2$ and $\ee>0$.
    Using systematically these two inequalities in \eqref{eq:PDIw}
    and taking $\ee$ small enough we finally obtain,
    \begin{equation*}
        \ptt w^{\dd} -\nu\Delta w^{\dd}
        + \nabla_xw^{\dd}\cdot  H^{\th}_p\lp x,\nabla_x u^{\dd}, \mu\rp 
        -2\nu\frac{\vp'}{\vp} \nabla_xw^{\dd}\cdot \nabla_xu^{\dd} 
        \\
        \leq
        C_{\ee}\lp \th+
        \th^{1-q'}\LL_{q'}\lp\mu\rp^{q'}\rp
        +\norminf{R^{\dd}},
    \end{equation*}
    where $C_{\ee}$ is a constant which depends on $\ee$ and
    the constants in the assumtions.
    From \ref{hypo:Monogregx}, the initial condition of $w^{\dd}$
    is bounded.
    Therefore the maximum principle for second-order parabolic
    equations implies that 
    \begin{equation}
        \label{eq:ineqwdd}
        \norminf{w^{\dd}}
        \leq
        C_{\ee}\lp \th+ \th T+
        \th^{1-q'}\int_0^T\LL_{q'}\lp\mu(t)\rp^{q'}dt\rp
        +T\norminf{R^{\dd}}.
    \end{equation}
    Let us point out that $\nabla_xu$ is the solution
    of the following backward $d$-dimensional parabolic equation,
    \begin{equation*}
        -\ptt\nabla_xu
        -\nu\Delta\nabla_xu
        +D^2_{x,x}u H_p\lp x,\nabla_xu,\mu\rp
        =
        \nabla_xf(x,m)
        -H_x\lp x,\nabla_xu,\mu\rp,
    \end{equation*}
    which has bounded coefficients and right-hand side,
    and a terminal condition in $C^{1+\bb_0}\lp\TT^d_a\rp$.
    Theorem $6.48$ in \cite{MR1465184} states
    that $\nabla_xu$ and $D^2_{x,x}u$ are continuous.
    This and the continuity of $H^{\th}$ and $H^{\th}_x$
    stated in Lemma \ref{lem:Hbound} imply that
    $R^{\dd}$ is uniformly convergent to $0$
    when $\dd$ tends to $0$.
    We conclude this step of the proof by passing to the limit
    in \eqref{eq:ineqwdd} as $\dd$ tends to $0$,
    using the estimate on
    $\int_0^T\LL_{q'}\lp\mu(t)\rp^{q'}dt$
    computed in the first step.
    We obtain that $\nabla_xu$ is uniformly bounded by a constant
    which depends on the constants in the assumptions,
    and depends linearly on $\th^{\frac12}$.

    \emph{Fourth step: obtaining uniform estimates
        on $\LL_{q'}\lp\mu\rp$ and $\LL_{\infty}\lp\mu\rp$.}
    
    %Let us notice that
    %\eqref{eq:aprioriLq'}
    %holds if we assume
    %in Lemma \ref{lem:apriorimu}
    %that the inequalities \eqref{eq:Lcoerbis}
    %and \eqref{eq:Lboundbis}
    %are satisfied only for
    %the probability measures $\mu$
    %of the form $\mu= m\otimes\dd_0$.
    Repeating the calculation in the proof
    of Lemma \ref{lem:apriorimu} with $L$
    satisfying 
    \eqref{eq:Lthcoer} and \eqref{eq:Lthbound},
    we obtain:
    \begin{equation}
        \label{eq:apriorimuLq'}
        \LL_{q'}\lp\mu(t)\rp^{q'}
        \leq
        4C_0^2\th^{q'}
        +\frac{\lp q'\rp^{q-1}\lp2C_0\rp^q}{q}
        \th^{q'}\norm[q]{\nabla_xu(t)}{L^q(m(t))}.
    \end{equation}
    This and the third step of this proof
    yield that $\sup_{t\in[0,T]}\LL_{q'}\lp\mu(t)\rp\leq C\th$
    for some $C$ depending only on the constants
    of the assumptions.
    We conclude that $\sup_{t\in[0,T]}\LL_{\infty}\lp\mu(t)\rp$
    satisfies a similar inequality
    using \eqref{eq:Hthpbound}.
\end{proof}

\section{Existence and Uniqueness Results}
\label{sec:MainMono}
Paragraph \ref{subsec:Exittda} is devoted to proving
the existence of solutions to \eqref{eq:MFGCttda},
which is step \ref{step:ttdaexi}.
In paragraph \ref{subsec:rd},
we propose a method to extend the existence
result to system \eqref{eq:MFGCrd} which is stated
on $\RR^d$; this concludes step \ref{step:MFGCrdexi}.
This method relies on compactness results
using the uniform estimates of $\nabla_xu$ that we
obtained in Lemma \ref{lem:Monoapriori}.
In paragraph \ref{subsec:UniMono}, we prove
step \ref{step:MFGCrduni},
namely
the uniqueness of the solution to \eqref{eq:MFGCrd}
and \eqref{eq:MFGCttda}.
Then the main results of the paper
and step \ref{step:MFGCb}
are addressed in paragraph \ref{subsec:Exib}.
We introduce a one-to-one correspondance between solutions
to \eqref{eq:MFGCb} and \eqref{eq:MFGCrd},
which allows us to
obtain directly the existence and the uniqueness
of the solution to \eqref{eq:MFGCb}
from the ones to \eqref{eq:MFGCrd}.

\subsection{Proof of Theorem 
\ref{thm:ExiMonottda}:
existence of solutions to \eqref{eq:MFGCttda}}
\label{subsec:Exittda}
We will use the a priori estimates
stated in Section \ref{sec:aprioriMono}
and the latter fixed point theorem,
in order to achieve step \ref{step:ttdaexi}
and prove the existence of solutions
to \eqref{eq:MFGCttda}.

\begin{proof}[Proof of Theorem \ref{thm:ExiMonottda}]
    We would like to use the 
    Leray-Schauder theorem \ref{thm:Leray}
    on a map which takes a flow of measures
    $\lp\mt_t\rp_{t\in[0,T]}\in\lp\cP\lp\TT^d_a\rp\rp^{[0,T]}$
    as an argument.
    However, $\cP\lp\TT^d_a\rp$ is not a Banach space.
    A way to go through this difficulty
    is to compose the latter map with a continuous map
    from a convenient Banach space to the set of
    such flows of measures.
    Here, we consider the map introduced in \cite{2019arXiv190205461F},
    namely $\rho:C^0\lp[0,T]\times\TT^d_a;\RR\rp
    \to C^0\lp[0,T]\times\TT^d_a;\RR\rp$
    defined by
    \begin{equation*}
        \rho(\mt)(t,x)
        =
        \frac{\mt_+(t,x)- a^{-d}\int\mt_+(t,y)dy}{\max\lp1,\int \mt_+\lp t,y\rp dy\rp}
        +a^{-d},
        %-\frac{\int\mt_+(t,y)dy}{a^d\max\lp1,\int \mt_+\lp t,y\rp dy\rp},
    \end{equation*}
    where $\mt_+(t,x)=\max\lp 0,\mt(t,x)\rp$.
    We will also have the use of $\mt^0$
    defined as the unique weak solution of
    \begin{equation}
        \label{eq:defmt0}
        \ptt \mt^0
        -\nu\Delta \mt^0
        =
        0
        \text{ on }
        (0,T)\times\TT^d_a,
        \quad
        \text{ and }
        \mt^0(0,\cdot)=m^0.
    \end{equation}
    We are now ready to construct the map $\Psi$
    on which we will use
    the Leray-Schauder theorem \ref{thm:Leray}.
    Take
    $\th\in[0,1]$,
    $u\in C^{0,1}\lp[0,T]\times\TT^d_a;\RR\rp$
    and
    $\mt\in C^{0}\lp[0,T]\times\TT^d_a;\RR\rp$.
    We define $m=\rho\lp\mt+\mt^0\rp$
    and
    $\lp\mu,\aa\rp\in C^0\lp[0,T];\cP\lp\TT^d_a\times\RR^d\rp\rp
    \times C^0\lp [0,T]\times\TT^d_a;\RR^d\rp$
    by,
    \begin{align*}
        \aa(t,x)
        &=
        -H^{\th}_p\lp t,x,\nabla_xu(t,x),\mu(t)\rp
        \\
        \mu(t)
        &=
        \lp I_d,\aa(t,\cdot)\rp\# m(t).
    \end{align*}
    This definition
    comes from the conclusions
    of Lemma \ref{lem:FPmucont} when $\th>0$.
    For $\th=0$, it simply consists in taking 
    $\aa=0$ and $\mu(t)=m(t)\otimes\dd_0$.
    Here we can repeat the calculation
    and obtain inequality \eqref{eq:apriorimuLq'}.
    This and \eqref{eq:Hthpbound}
    implies that
    $\norminf{\aa}$
    is bounded by $C\th$
    for some constant $C>0$ which depends on
    $\norminf{\nabla_xu}$ and is independent
    of $\th$ and $a$.

    Then we define
    %$\mo\in C^{0}\lp[0,T]\times\TT^d_a;\RR\rp$
    $\mo$ the solution in the sense of distributions of
    \begin{equation*}
        \ptt \mo
        -\nu\Delta \mo
        +\divo\lp\aa\mo\rp
        =
        0,
    \end{equation*}
    supplemented with the initial condition
    $\mo(0,\cdot)=m^0$, with $m^0$ being $\bb_0$-Hölder
    continuous.
    Theorem $2.1$ section $V.2$ in \cite{MR0241822}
    states that $\mo$ is uniformly bounded by a constant
    which depends on $\norminf{m_0}$
    and $\norminf{\aa}$.
    %This and \eqref{eq:Hpbound}
    %yield that $\aa\mo$
    %is bounded by a constant which depends on
    %$\norminf{m_0}$, $\norminf{\nabla_xu}$
    %and the constants in the assumptions.
    Theorem $6.29$ in \cite{MR1465184}
    yields that
    $m\in C^{\frac{\bb}2,\bb}\lp[0,T]\times\TT^d_a\rp$
    for $\bb\in(0,\bb_0)$,
    and that its associated norm can be estimated from
    above by a constant which depends on
    $\norminf{\nabla_xu}$, $\bb$, $a$
    and the constants in the assumptions.
    The same arguments applied to $\mt^0$
    defined in \eqref{eq:defmt0} imply that 
    $\mt^0$ is in $C^{\frac{\bb}2,\bb}\lp[0,T]\times\TT^d_a\rp$
    and its associated norm is bounded.

    Then we take
    $\muo(t)=\lp I_d,\aa(t,\cdot)\rp\#\mo(t)$
    for any $t\in[0,T]$,
    and
    $\uo\in C^{0,1}\lp[0,T]\times\TT^d_a;\RR\rp$
    the unique solution in
    the sense of distributions of
    the following heat equation
    with bounded right-hand side,
    \begin{equation*}
        -\ptt \uo
        -\nu\Delta \uo
        =
        -H^{\th}\lp t,x,\nabla_xu,\muo(t)\rp
        +\th f(x,\mo(t)),
    \end{equation*}
    supplemented with the terminal condition
    $\uo(T,\cdot)=\th g\lp\cdot,\mo(T)\rp$
    which is in $C^{1+\bb_0}\lp\TT^d_a\rp$.
    %The latter differential equation is 
    %the heat equation with a $L^{\infty}$
    %right-hand side.
    Classical results (see for example Theorem $6.48$ in \cite{MR1465184})
    state that $u$ is in $C^{\frac12+\frac{\bb}2,1+\bb}$
    and its associated norm is bounded 
    by a constant which depends 
    on $\norminf{\nabla_xu}$, $\bb$, $a$
    and the constants in the assumptions.
    
    We can now construct the map
    $\Psi:\lp\th, u,\mt\rp\mapsto\lp\uo,\mo-\mt^0\rp$,
    from $C^{0,1}\lp[0,T]\times\TT^d_a;\RR\rp
    \times C^{0}\lp[0,T]\times\TT^d_a;\RR^d\rp$
    into itself.
    This map is continuous and compact,
    it satisfies $\Psi\lp0,u,\mt\rp=0$
    for any $\lp u,\mt\rp$.
    In particular,
    the fact that 
    $\norminf{\aa}\leq C\th$
    in the previous paragraph,
    implies that
    $\mo$ tends to $\mt^0$
    and $\uo$ tends to $0$
    as $\th$ tends to $0$.
    This gives the continuity
    of $\Psi$ at $\th=0$.
    Moreover the fixed points of $\Psi(\th)$
    are exactly the solutions to
    \eqref{eq:MFGCth},
    which are uniformly bounded
    by Lemma \eqref{lem:Monoapriori}.
    Therefore, by the Leray-Schauder fixed point theorem
    \ref{thm:Leray},
    there exists a solution to
    \eqref{eq:MFGCttda}.
\end{proof}

\subsection{Proof of Theorem \ref{thm:Exird}:
    passing from the torus to $\RR^d$}
\label{subsec:rd}
The purpose of this paragraph is to
extend the existence result 
to the system \eqref{eq:MFGCrd}
and achieve step \ref{step:MFGCrdexi}.

\begin{proof}[Proof of Theorem \ref{thm:Exird}]
    \emph{First step: constructing a sequence of approximate solutions.}

For $a>0$ we define
$\mt^{0,a}=\pi^a\#m^0$,
where $\pi^a:\RR^d\to\TT_a^d$ is
the quotient map.
Let $\chi^a:\TT^1_a\to\RR$ be the canonical injection
from the one-dimensional torus of radius $a$
to $\RR$, which image is $\lb-\frac{a}2,\frac{a}2\rp$.
Take $\psit\in C^{2}\lp\RR;\RR\rp$
periodic with a period equal to $1$ and
such that,
\begin{equation}
    \label{eq:defpsi}
\begin{aligned}
    \psit(x)
    &=
    x,
    \quad
    \text{ if }
    |x|\leq\frac{1}4,
    \\
    \labs\psit(x)\rabs
    &\leq
    \labs x\rabs,
    \quad
    \text{ for any }
    x\in\lb-\frac12,\frac12\rb,
    %\\
    %\psi(1)
    %&=
    %\psi(-1)
    %=
    %0,
    %\\
    %\psi'(1)
    %&=
    %\psi'(-1)
    %=
    %0.
\end{aligned}
\end{equation}
We define $\psi^a:\TT^d_a\to\RR^d$
by $\psi^a(x)_i
=a\psit\lp a^{-1}\chi^a(x_i)\rp$
for $i=1,\dots,d$,
this is a $C^2$ function.
Since $\psit\lp \frac{\cdot}a\rp$ has a period of $a$,
the function $\psi^a\circ\pi^a:\RR^d\to\RR^d$ satisfies
\begin{equation}
    \label{eq:psipi}
    \psi^a\circ\pi^a(x)_i=a\psit\lp\frac{x_i}a\rp,
\end{equation}
for $i=1,\dots,d$
and $x\in\RR^d$,
and is a $C^2$ function.
\begin{comment}
in particular it is also a $C^2$ function
and its first order derivatives are 
\begin{equation}
    \label{eq:psipsi}
    \partial_{x_i}\lp\psi^a\circ\pi^a\rp(x)=\psit'\lp a^{-1}\chi^a\lp\pi^a(x)\rp_i\rp
\end{equation}
Let us mention that for any $x\in\RR^d$,
we have
\begin{equation}
    \label{eq:psipi3}
    \labs\psi^a\circ\pi^a(x)\rabs
    \leq
    \labs x\rabs,
\end{equation}
moreover $\psi^a\circ\pi^a:\RR^d\to\RR^d$
is differential at any $x\in\RR^d\backslash\lp\frac{a}2+a\ZZ\rp^d$
and its derivatives are 
$\partial_{x_i}\lp\psi^a\circ\pi^a\rp(x)=\psit'\lp a^{-1}\chi^a\lp\pi^a(x)\rp_i\rp$
for $i=1,\dots,d$,
since $\psi^a\circ\pi^a$ is continuous at any $x\in\lp\frac{a}2+a\ZZ\rp^d$,
and that the right and left derivatives at these points are equal to $0$,
then $\psi^a\circ\pi^a\in C^{1}\lp\RR^d;\RR^d\rp$.
In particular, we have
and if $\labs x_i\rabs\leq \frac{a}2$ for any $i=1,\dots,d$,
then
\begin{equation}
    \label{eq:psipi2}
    \labs\psi^a\circ\pi^a(x)
    -\psi^a\circ\pi^a(y)\rabs
    \leq
    \norminf{\psit'}
    |x-y|.
\end{equation}
\end{comment}
We are ready to construct periodic approximations
of $L$, $f$ and $g$ defined by,
\begin{align*}
    L^a\lp t,x,\aa,\mu\rp
    &=
    L\lp t,\psi^a(x),\aa,\lp \psi^a\otimes I_d\rp\# \mu\rp,
    \\
    f^a\lp t,x,m\rp
    &=
    f\lp t,\psi^a(x),\psi^a\# m\rp,
    \\
    g^a\lp x,m\rp
    &=
    g\lp \psi^a(x),\psi^a\# m\rp,
\end{align*}
for $(t,x)\in[0,T]\times\TT^d_a$,
$\aa\in\RR^d$,
$\mu\in\cP\lp\TT^d_a\times\RR^d\rp$.
Let $H^a$ be the periodic Hamiltonian
associated with $L^a$ by the Legendre transform:
\begin{equation*}
    H^a\lp t,x,p,\mu\rp
    =
    H\lp t,\psi^a(x),p,\lp \psi^a\otimes I_d\rp\# \mu\rp.
\end{equation*}
Let us point out that the fact that
$L$, $H$, $f$ and $g$ satisfy
\ref{hypo:Lreg}-\ref{hypo:Monogregx},
%\ref{hypo:Lstrconv}-\ref{hypo:Llipmu},
implies that
$L^{a}$, $H^a$, $f^a$ and $g^a$
satisfy these assumptions too
with $C_0\norminf{\psit'}$ instead of $C_0$.
\begin{comment}
    %% passage avec le passage de l'ancienne hypothèse Exi2 au cas périodique.
Let us check that 
\ref{hypo:MonoInvert}
is satisfied for $\psi$ defined by
$\psi=\pi^a\circ\psi^a$.
For $\pt\in C^{0}\lp[0,T]\times\TT^d_a;\RR^d\rp$
and $\mt\in C^{0}\lp[0,T];\cP\lp \TT^d_a\rp\rp$,
we define 
$\pt\in C^{0}\lp[0,T]\times\RR^d;\RR^d\rp$
and $\mt\in C^{0}\lp[0,T];\cP\lp \RR^d\rp\rp$
by $p(t,x)=\pt\lp t,\pi^a(x)\rp$
and
$m(t)=\psi^a\#\mt$,
for $(t,x)\in[0,T]\times\RR^d$.
Since $L$ and $H$ satisfy
\ref{hypo:MonoInvert},
there exists $\aa(t)\in L^{\infty}\lp m(t)\rp$
for any $t\in[0,T]$,
and $\mu\in C^0\lp[0,T];\cP\lp\RR^d\times\RR^d\rp\rp$
satisfying \eqref{eq:FPwoutb}.
We define $\aat(t)\in L^{\infty}\lp \mt(t)\rp$
for any $t\in[0,T]$,
and $\mut\in C^0\lp[0,T];\cP\lp\TT^d_a\times\RR^d\rp\rp$
by $\aat(t,x)=\aa(t,\psi^a(x))$
and $\mut(t)=\lp I_d,\aat(t)\rp\#\mt(t)$.
Therefore we have,
\begin{equation*}
    \lp \psi^a\otimes I_d\rp\#\mut(t)
    =
    \lp \psi^a, \aat(t)\rp\#\mt(t)
    =
    \lp I_d, \aa(t)\rp\#m(t)
    =
    \mu(t).
\end{equation*}
This and the definition of $H^a$
imply that
\begin{align*}
    \aat(t,x)
    &=
    \aa\lp t,\psi^a(x)\rp
    \\
    &=
    -H_p\lp t,\psi^a(x),p\lp t,\psi^a(x)\rp,\mu(t)\rp
    \\
    &=
    -H^a_p\lp t,x,\pt\lp t,\pi^a\circ\psi^a(x)\rp,\mut(t)\rp.
\end{align*}
Therefore $L^a$ and $H^a$
satisfy \ref{hypo:MonoInvert}
with $\psi=\pi^a\circ\psi^a$.
\todo{changer une nouvelle fois avec la version
    de MonoInvert avec le $\psi$}
\end{comment}
So we can define $\lp \ut^a,\mt^a,\mut^a\rp$
a solution to \eqref{eq:MFGCttda}
with $H^a$, $f^a$, $g^a$ and $\mt^{0,a}$
instead of $H$, $f$, $g$ and $m^0$.
We define
$u^a\in C^0\lp[0,T]\times\RR^d;\RR\rp$,
 $m^a\in C^0\lp[0,T];\cP\lp\RR^d\rp\rp$
and $\mu^a\in C^0\lp[0,T];\cP\lp\RR^d\times\RR^d\rp\rp$
respectively by
\begin{align*}
    u^a(t,x)
    =
    \ut^a\lp t,\pi^a(x)\rp,
    \quad
    m^a(t)
    =
    \psi^a\#\mt^a(t),
    \text{ and }
    \quad
    \mu^a(t)
    =
    \lp\psi^a\otimes I_d\rp\#\mut^a(t),
\end{align*}
for $(t,x)\in[0,T]\times\RR^d$.

\emph{Second step: Proving that $m^a$ is compact.}

We are going to use the Arzelà-Ascoli Theorem
on $C^0\lp[0,T];\lp\cP\lp\RR^d\rp,W_1\rp\rp$
($\cP\lp\RR^d\rp$ is endowed with the
$1$-Wassertein distance).
First we prove that for any $t\in[0,T]$,
the sequence $\lp m^a(t)\rp_{a>1}$
is compact with the $1$-Wassertein distance,
by proving that $\int_{\RR^d}|x|^2dm^{a}(t,x)$
is uniformly bounded in $a$.
At time $t=0$, we have
\begin{equation*}
    \int_{\RR^d}\labs x\rabs^2 dm^a(0,x)
    =
    \int_{\TT^d_a}\labs \psi^a(x)\rabs^2 d\mt^{a,0}(x)
    =
    \int_{\RR^d}\labs \psi^a\circ\pi^a(x)\rabs^2 dm^0(x)
    \leq
    \int_{\RR^d}\labs x\rabs^2 dm^0(x)
    \leq
    C_0,
\end{equation*}
using \eqref{eq:defpsi}, \eqref{eq:psipi} and \ref{hypo:Monogregx}.
Let us differentiate
$\int_{\RR^d}|x|^2dm^{a}(t,x)$
with respect to time, perform
some integrations by part and obtain that
\begin{align*}
    \frac{d}{dt}
    \int_{\RR^d}|x|^2dm^{a}(t,x)
    &=
    \frac{d}{dt}
    \int_{\TT^d_a}\labs\psi^a(x)\rabs^2d\mt^{a}(t,x)
    \\
    &=
    \int_{\TT^d_a}
    \labs\psi^a(x)\rabs^2
    \lp\nu\Delta\mt^{a}(t,x)
    -\divo\lp\aa^{\mut^a(t)}(x)\mt^a(t,x)\rp\rp
    dx
    \\
    &\!\begin{multlined}[t][10.5cm]
        =
        2\int_{\TT^d_a}
        \sum_{i=1}^d
        \lb
        \nu\psit''\lp\frac{\chi^a(x^i)}a\rp
        \psit\lp\frac{\chi^a(x^i)}a\rp
        +\nu\psit'\lp\frac{\chi^a(x^i)}a\rp^2
        \right.
        \\
        \left.
        +\psi^a\lp x\rp
        \psit'\lp\frac{\chi^a(x^i)}a\rp
        \aa^{\mut^a(t),i}(x)
        \rb d\mt^a(t,x)
    \end{multlined}
    \\
    &\leq
    2\nu d\norminf{\psit''}\norminf{\psit}
    +2\nu d\norminf[2]{\psit'}
    +\norminf[2]{\psi'}\norminf[2]{\aa^{\mut^a(t)}}
    +\int_{\TT^d_a}\labs\psi^a(x)\rabs^2d\mt^{a}(t,x)
    \\
    &\leq
    2\nu d\norminf{\psit''}\norminf{\psit}
    +2\nu d\norminf[2]{\psit'}
    +\norminf[2]{\psit'}\norminf[2]{\aa^{\mut^a(t)}}
    +\int_{\RR^d}\labs x\rabs^2dm^{a}(t,x).
\end{align*}
We recall that
$(t,x)\mapsto\aa^{\mut^a(t)}(x)$
is uniformly bounded
with respect to $t$ and $a$
by Lemma \ref{lem:Monoapriori}.
Therefore, the latter two inequalities and
a comparison principle for ordinary differential equation
imply that 
$\int_{\RR^d}|x|^2dm^{a}(t,x)$
is uniformly bounded with respect to $a$ and $t$.

We define $\Xo^a$ a random process on $\RR^d$ by
\begin{equation*}
    d\Xo^a_t
    =
    \aa^{\mu^a(t)}\lp \pi^a\lp \Xo^a_t\rp\rp dt
    +\sqrt{2\nu}dB_t,
    \;
    \text{ and }
    \cL\lp \Xo^a_0\rp
    =
    m^0,
\end{equation*}
where $B$ is a Brownian motion on $\RR^d$
independent of $\Xo^a_0$.
%Here we only know that $(t,x)\mapsto\aa^{\mu^a(t)}(x)$
%is measurable and bounded,
%and we do not assume that it is Lipschitz continuous with respect to $x$,
%thus we cannot use the Cauchy-Lipschitz Theorem on stochastic differential
%equations to construct $\Xo^a$.
%However we can define $\Xo^a$ as a weak solution
%to the latter SDE using the Girsanov's Theorem,
%see \cite{MR3752669}
%paragraph $3.3.1$
%for instance.
%In particular, we can assume that $B$ and $\Xo_0^a$
%are independent.
For  $t,s\in[0,T]$, we have that,
\begin{align*}
    \EE\lb \labs \Xo^a_t-\Xo^a_s\rabs\rb
    &\leq
    \EE\lb \labs \Xo^a_t-\Xo^a_s\rabs^2\rb^{\frac12}
    \\
    &\leq
    \EE\lb\labs \int_s^t\sqrt{2\nu} dW_r\rabs^2\rb^{\frac12}
    +\EE\lb\labs \int_s^t\aa^{\mu^a(r)}dr\rabs^2\rb^{\frac12}
    \\
    &\leq
    \sqrt{2\nu d}|t-s|^{\frac12}
    +|t-s|\sup_{r\in[0,T]}\norminf{\aa^{\mu^a(r)}}.
\end{align*}
We define $\Xt^a_t=\pi^a\lp \Xo^a_t\rp\in\TT^d_a$
and $X^a_t=\psi^a\lp \Xt^a_t\rp\in\RR^d$,
for $t\in[0,T]$.
One may check that
the law of $\Xt^a_t$
satisfies the same Fokker-Planck equation
in the sense of distributions
as $\mt^a(t)$ by testing it with
$C^{\infty}\lp(0,T)\times\TT^d\rp$ test functions.
Therefore, the law of $\Xt^a_t$ is $\mt^a(t)$
and the law of $X^a_t$ 
is $m^a(t)$.
By definition of the $1$-Wassertein distance,
we obtain
\begin{align*}
    W_1\lp m^a(t),m^a(s)\rp
    &\leq
    \EE\lb \labs X^a_t-X^a_s\rabs\rb
    \\
    &\leq
    \EE\lb \labs \psi^a\circ\pi^a\lp\Xo^a_t\rp
    -\psi^a\circ\pi^a\lp\Xo^a_s\rp\rabs\rb
    \\
    &\leq
    \norminf{\psit'}
    \EE\lb \labs \Xo^a_t-\Xo^a_s\rabs\rb
    \\
    &\leq
    \norminf{\psit'}
    \lp\sqrt{2\nu d}|t-s|^{\frac12}
    +|t-s|\sup_{r\in[0,T]}\norminf{\aa^{\mu^a(r)}}\rp,
\end{align*}
where we used \eqref{eq:psipi} and the mean value theorem
to pass from the second to the third line
in the latter chain of inequalities.
Therefore by the Arzelà-Ascoli theorem,
$\lp m^{a}\rp_{a>0}$ is relatively compact
in $C^0\lp[0,T];\lp\cP\lp\RR^d\rp,W_1\rp\rp$.

\emph{Third Step: passing to the limit for a subsequence.}

We recall that $\ut^a$ and $\nabla_x\ut^a$
are uniformly bounded with respect to $a$,
so are $u^a$ and $\nabla_xu^a$.
Moreover $u^a$ satisfies the following
PDE,
\begin{equation*}
    -\ptt u^a
    -\nu\Delta u^a
    +H\lp t,\psi^a\circ\pi^a(x),\nabla_xu(t,x),\mu^a(t)\rp
    =
    f\lp t,\psi^a\circ\pi^a(x),m^a(t)\rp,
\end{equation*}
for $(t,x)\in (0,T)\times B_{\RR^d}\lp 0,a\rp$,
we recall that $\psi^a\circ\pi^a(x)=x$
if $|x|\leq\frac{a}4$.
For $a_0>0$,
we choose $a$ such that $a>4\lp a_0+1\rp$,
this implies that $u^a$
satisfies a backward heat equation on
$B_{\RR^d}\lp 0,a_0+1\rp$
with a bounded right-hand side, a bounded terminal condition,
and bounded boundary conditions.
Classical results on the heat equation
(see for example Theorem $6.48$ in \cite{MR1465184})
state that $u^a$ is in
$C^{\frac12+\frac{\bb}2,1+\bb}\lp [0,T]\times B_{\RR^d}\lp0,a_0\rp;\RR\rp$
and that its associated norm is bounded by a constant
which depends on the constants in the assumptions
and $a_0$, but not on $a$.
Therefore 
$\lp u^a_{|B_{\RR^d}\lp 0,a_0\rp}\rp_{a>1}$
is a compact sequence in
$C^{0,1}\lp[0,T]\times B_{\RR^d}\lp0,a_0\rp;\RR\rp$
for any $a_0>0$.
Then by a diagonal extraction method,
there exists $a_n$ an increasing sequence
tending to $+\infty$
in $\RR_+$ such that 
\begin{align*}
    m^{a_n}
    &\to
    m
    \quad
    &\text{ in } C^0\lp[0,T],\lp\cP\lp\RR^d\rp,W_1\rp\rp,
    \\
    %\lp\chi_a,I_d\rp\#
    %\mu^{a_n}
    %&\to
    %\mu
    %\quad
    %&\text{ weakly in } \cP\lp\RR^d\times\RR^d\rp,
    %\\
    u^{a_n}
    &\to
    u
    \quad
    &\text{ locally in }
    C^{0,1},
\end{align*}
for some $(u,m)\in
C^{0,1}\lp[0,T]\times\RR^d;\RR\rp
\times
C^0\lp[0,T];\lp\cP\lp\RR^d\rp,W_1\rp\rp$.
Let us prove that 
for $t\in[0,T]$,
$\mu^{a_n}(t)$
converges to a fixed
point of \eqref{eq:murd}
when $n$ tends to infinity;
indeed we notice that
\begin{align*}
    \mu^{a_n}(t)
    &=
    \lp \psi^{a_n}\otimes I_d\rp\#\mut^{a_n}(t)
    \\
    &=
    \lp \psi^{a_n}\otimes I_d\rp\#
    \lb\lp I_d,
    -H^{a_n}_p\lp t,\cdot,\nabla_x\ut^{a_n}\lp t,\pi^{a_n}\circ\psi^{a_n}(\cdot)\rp,\mut^{a_n}(t)\rp
    \rp\#\mt^{a_n}\rb
    \\
    &=
    \lp \psi^{a_n},
    -H_p\lp t,\psi^{a_n}(\cdot),\nabla_xu^{a_n}\lp t,\psi^{a_n}(\cdot)\rp,\mu^{a_n}(t)\rp
    \rp\#\mt^{a_n}
    \\
    &=
    \lp I_d,
    -H_p\lp t,\cdot,\nabla_xu^{a_n}\lp t,\cdot\rp,\mu^{a_n}(t)\rp
    \rp\#m^{a_n}.
\end{align*}
In particular, $\aa^{\mut^{a_n}(t)}=\aa^{\mu^{a_n}(t)}\circ\psi^{a_n}$
so $\norm{\aa^{\mu^{a_n}(t)}}{L^{\infty}\lp m\rp}$
is not larger than
$\norm{\aa^{\mut^{a_n}(t)}}{L^{\infty}\lp \mt\rp}$
since the support of $m^{a_n}$ is contained
in the image of the support of $\mt^{a_n}$
by $\psi^{a_n}$.
We proved in the previous step that $\lp m^a(t)\rp_{a\geq 1}$ is
compact in
$\lp\cP\lp\RR^d\rp,W_1\rp$,
and so is $\lp\mu^{a_n}(t)\rp_{n\geq1}$
in $\lp\cP\lp\RR^d\times\RR^d\rp,W_1\rp$,
since they are the pushforward measures
of $\lp m^{a_n}(t)\rp_{n\geq1}$ by
$\lp I_d,\aa^{\mu^{a_n}(t)}\rp$.
Let $\mu(t)\in\cP\lp\RR^d\times\RR^d\rp$
be the limit of a convergent subsequence of 
$\lp\mu^{a_n}(t)\rp_{n\geq1}$.
Passing to the limit in the weak* topology
in the latter chain of equalities
implies that
\begin{equation*}
    \mu(t)=
    \lp I_d,
    -H_p\lp t,\cdot,\nabla_xu\lp t,\cdot\rp,\mu(t)\rp
    \rp\#m(t).
\end{equation*}
Moreover, the uniqueness
of the fixed point \ref{eq:murd}
holds here,
see \cite{MR3805247} Lemma $5.2$ for the proof.
We obtained that there exists a unique
fixed point satisfying \ref{eq:murd},
and that it is the limit of any convergent
subsequence of $\lp\mu^{a_n}(t)\rp$.
This implies that the whole sequence
$\lp\mu^{a_n}(t)\rp_{n\geq1}$
tends to $\mu(t)$
in $\lp\cP\lp\RR^d\times\RR^d\rp,W_1\rp$.

Let us point out that $m^{a_n}$ satisfies
\begin{equation*}
    \ptt m^{a_n}
    -\nu\Delta m^{a_n}
    -\divo\lp H_p\lp t,x,\nabla_xu^{a_n},\mu^{a_n}\rp m^{a_n}\rp
    =
    0
\end{equation*}
in the sense of distributions on $(0,T)\times B\lp0,\frac{a_n}4\rp$,
by the definitions of $\psi^a$ and $\psit$.
Furthermore, at time $t=0$ we know that
$m^{a_n}(0)=\lp\psi^{a_n}\circ\pi^{a_n}\rp\#m^0$.
We recall that $\psi^{a_n}\circ\pi^{a_n}(x)=x$
for $x\in B_{\RR^d}\lp 0,\frac{a_n}4\rp$.
This implies that $m^{a_n}(0)$ tends to $m^0$
in the weak* topology of $\cP\lp \RR^d\rp$.
%Moreover we recall that  
%$\lp m^{a_n}(0)\rp_{n\in\NN}$
%is compact in $\cP\lp\RR^d\rp$ equipped with the
%weak topology, therefore 
%$m^{a_n}(0)$ is weakly convergent to $m^0$.

Finally we obtain that $(u,m,\mu)$ is a solution to \eqref{eq:MFGCrd},
by passing to the limit as $n$ tends to infinity
in the equations satisfied by
$\lp u^{a_n},m^{a_n},\mu^{a_n}\rp$.
\end{proof}

\begin{remark}
    In the above proof,
    we obtain that there exists a unique
    fixed point satisfying \ref{eq:murd}.
    We have thereby extended
    the conclusions of Lemma
    \ref{lem:FPmuLS} to system
    \ref{eq:MFGCrd}.
    Similarly, one may extend the conclusions
    of Lemma \ref{lem:FPmucont}
    to system \eqref{eq:MFGCrd}.
\end{remark}

\subsection{Proof of Theorem \ref{thm:UniMono}:
    uniqueness of the solutions to \eqref{eq:MFGCrd} and \eqref{eq:MFGCttda}}
\label{subsec:UniMono}
Step \ref{step:MFGCrduni},
namely the uniqueness of the solution to \eqref{eq:MFGCrd},
is obtained from the monotonicity assumptions
\ref{hypo:LMono} and \ref{hypo:gMono},
and the same arguments as in the case
of MFG without interaction through controls.
%The following uniqueness results and Theorem
%\ref{thm:UniMonob} below are the only
%results in which we assume
%\ref{hypo:gMono}.

\begin{proof}[Proof of Theorem \ref{thm:UniMono}]
    Here, we write the proof for
    the system \eqref{eq:MFGCrd}.
    However, none of the arguments below is
    specific to the domain $\RR^d$,
    therefore this proof can be repeated
    for \eqref{eq:MFGCttda}.

    We suppose that $(u^1,m^1,\mu^1)$ 
    and $(u^2,m^2,\mu^2)$
    are two solutions to \eqref{eq:MFGCrd}.
    Now standard arguments (see \cite{MR2295621}) lead to
    \begin{multline}
        \label{eq:MonoUni}
        0
        =
        \int_0^T\int_{\RR^d}
        \lb \nabla_x(u^1-u^2)
        \cdot H_p\lp t, x,\nabla_xu^1,\mu^1\rp
        - H\lp t,x,\nabla_xu^1,\mu^1\rp
        +H\lp t,x,\nabla_xu^2,\mu^2\rp
        \rb dm^1(t,x)
        \\
        + \int_0^T\int_{\RR^d}
        \lb \nabla_x(u^2-u^1)
        \cdot H_p\lp t,x,\nabla_xu^2,\mu^2\rp
        - H\lp t,x,\nabla_xu^2,\mu^2\rp
        +H\lp t,x,\nabla_xu^1,\mu^1\rp
        \rb dm^2(t,x)
        \\
        +  \int_0^T\int_{\RR^d}
        \lp f(t,x,m^1(t))-f(t,x,m^2(t))\rp
        d(m^1(t,x)-m^2(t,x))dt
        \\
        +  \int_{\RR^d}
        \lp g(x,m^1(T))-g(x,m^2(T))\rp
        d(m^1(T,x)-m^2(T,x)).
    \end{multline}
    Recall that
    \begin{equation}
        \label{eq:conjugacy}
        \begin{aligned}
            &L\lp t,x,\aa^{\mu^i},\mu^i\rp
            =
            \nabla_xu^i
            \cdot H_p\lp t,x,\nabla_xu^i,\mu^i\rp
            - H\lp t,x,\nabla_xu^i,\mu^i\rp,
            \\
            &\nabla_xu^i
            =
            -L_{\aa}\lp t,x,\aa^{\mu^i},\mu^i\rp,
        \end{aligned}
    \end{equation}
    because $L$ is the Legendre tranform of $H$.
    From \ref{hypo:gMono},
    \eqref{eq:MonoUni} and \eqref{eq:conjugacy},
    we obtain that,
    \begin{multline}
        \label{eq:UniMono_aux2}
        0
        \geq
        \int_0^T\int_{\RR^d}
        \lb L\lp t,x,\aa^{\mu^1},\mu^1\rp
        -L\lp t,x,\aa^{\mu^2},\mu^2\rp
        -\lp\aa^{\mu^1}-\aa^{\mu^2}\rp
        \cdot L_{\aa}\lp t,x,\aa^{\mu^2},\mu^2\rp
        \rb dm^1(t,x)dt \\
        +\int_0^T\int_{\RR^d}
        \lb L\lp t,x,\aa^{\mu^2},\mu^2\rp
        -L\lp t,x,\aa^{\mu^1},\mu^1\rp
        -\lp\aa^{\mu^2}-\aa^{\mu^1}\rp
        \cdot L_{\aa}\lp t,x,\aa^{\mu^1},\mu^1\rp
        \rb dm^2(t,x)dt
    \end{multline}
    The function $L$ is strictly convex in $\aa$
    by Lemma \ref{lem:Lstrconv},
    which implies that,
    \begin{equation}
        \label{eq:Lconvex}
    \begin{aligned}
        L\lp t,x,\aa^{\mu^1},\mu^2\rp
        -L\lp t,x,\aa^{\mu^2},\mu^2\rp
        -\lp\aa^{\mu^1}-\aa^{\mu^2}\rp
        \cdot L_{\aa}\lp t,x,\aa^{\mu^2},\mu^2\rp
        \geq
        0,
        \\
        L\lp t,x,\aa^{\mu^2},\mu^1\rp
        -L\lp t,x,\aa^{\mu^1},\mu^1\rp
        -\lp\aa^{\mu^2}-\aa^{\mu^1}\rp
        \cdot L_{\aa}\lp t,x,\aa^{\mu^1},\mu^1\rp
        \geq
        0,
    \end{aligned}
    \end{equation}
    and \eqref{eq:Lconvex} turn to identities if and only if 
    $\aa^{\mu^1}=\aa^{\mu^2}$.
    The latter inequalities 
    and \eqref{eq:UniMono_aux2}
    yield
    \begin{align*}
        0
        &\geq
        \int_0^T\int_{\RR^d}
        \lb L\lp t,x,\aa^{\mu^1},\mu^1\rp
        -L\lp t,x,\aa^{\mu^1},\mu^2\rp
        \rb dm^1dt
        +\int_0^T\int_{\RR^d}
        \lb L\lp t,x,\aa^{\mu^2},\mu^2\rp
        -L\lp t,x,\aa^{\mu^2},\mu^1\rp
        \rb dm^2dt \\
        &=
        \int_0^T\int_{\RR^d\times\RR^d}
        \lb L\lp t,x,\aa,\mu^1\rp
        -L\lp t,x,\aa,\mu^2\rp
        \rb d\lp \mu^1-\mu^2\rp
        (t,x,\aa)dt.
    \end{align*}
    Assumption \ref{hypo:LMono}
    turns the latter inequality into an equality.
    This,
    the case of equality in \eqref{eq:Lconvex} and
    the continuity of $\aa^{\mu^1}$ and $\aa^{\mu^2}$
    yield that
    $\aa^{\mu^1}=\aa^{\mu^2}$.
    %Therefore,
    %$m^1$ and $m^2$
    %solve the same FPK equation
    %with the same initial condition,
    This implies that $m^1=m^2$
    by the uniqueness of the solution
    to \eqref{eq:FPKrd}, \eqref{eq:CImrd}.
    Therefore, we obtain $\mu^1=\mu^2$,
    and then $u^1=u^2$ by the uniqueness of the solution
    to \eqref{eq:HJBrd},\ref{eq:CFurd}.
\end{proof}

\subsection{Theorems \ref{thm:UniMonob} and \ref{thm:bpower}:
    existence and uniqueness of the solution to \eqref{eq:MFGCb}}
\label{subsec:Exib}
So far, no distinction has been made between $\mu_b$
and $\mu_{\aa}$, because they coincide for \eqref{eq:MFGCrd}
and \eqref{eq:MFGCttda}.
Now 
they may differ since the drift function
and the control may be different.
%We chose the index $\aa$ in $\mu_{\aa}$ to stress out
%that at the Mean Field equilibrium, $\mu_{\aa}$
%should be the joint law of the states and the controls.
In this case $\mu_b$ defined by
    $\mu_b(t)
    =
    \Bigl[
    \lp x,\aa\rp
    \mapsto
    \lp x,
    b\lp t,x,\aa,\mu_{\aa}(t)\rp\rp
    \Bigr]{\#}\mu_{\aa}(t)$
is naturally the joint law of the states and the drifts.
The idea here to pass from \eqref{eq:MFGCrd} to \eqref{eq:MFGCb},
is to assume that $b$ is invertible with respect
to $\aa$, which changes the optimization problem
in $\aa$ into a new optimization problem
expressed in term of $b$.
This consists in changing the Lagrangian
from
$L\lp t,x,\aa,\mu_{\aa}\rp$
into 
\begin{equation*}
    L^b\lp t,x,b,\mu_{b}\rp
    =
    L\lp t,x,\aa^*\lp t,x,b,\mu_b\rp,
    \Bigl[
    \lp x,\bt\rp
    \mapsto
    \lp x,
    \aa^*\lp t,x,\bt,\mu_{b}\rp\rp
    \Bigr]{\#}\mu_{b}\rp.
\end{equation*}
The Hamiltonian $H^b$
defined as the Legendre transform of
$L^b$ is given by
\begin{equation}
    \label{eq:defHb}
    H^b\lp t,x,p,\mu_{b}\rp
    =
    H\lp t,x,p,
    \Bigl[
    \lp x,\bt\rp
    \mapsto
    \lp x,
    \aa^*\lp t,x,\bt,\mu_{b}\rp\rp
    \Bigr]{\#}\mu_{b}\rp.
\end{equation}
Conversely, we can obtain $L$ and $H$ from
$L^b$ and $H^b$ with the following relations,
\begin{align*}
    L\lp t,x,\aa,\mu_{\aa}\rp
    &=
    L^b\lp t,x,b\lp t,x,\aa,\mu_{\aa}\rp,
    \Bigl[
    \lp x,\aa\rp
    \mapsto
    \lp x,
    b\lp t,x,\aa,\mu_{\aa}\rp\rp
    \Bigr]{\#}\mu_{\aa}\rp,
    \\
    H\lp t,x,p,\mu_{\aa}\rp
    &=
    H^b\lp t,x,p,
    \Bigl[
    \lp x,\aa\rp
    \mapsto
    \lp x,
    b\lp t,x,\aa,\mu_{\aa}\rp\rp
    \Bigr]{\#}\mu_{\aa}\rp.
\end{align*}
Now we can state the following lemma which allows us to
pass from \eqref{eq:MFGCrd} to \eqref{eq:MFGCb},
or vice versa.
\begin{lemma}
    \label{lem:equisol}
    Under assumption
    \ref{hypo:binvert},
    $\lp u,m,\mu_{\aa},\mu_{b}\rp$
    is a solution to \eqref{eq:MFGCb}
    if and only if
    $\lp u,m,\mu_{b}\rp$
    is a solution to \eqref{eq:MFGCrd}
    with $H^b$ instead of $H$.
\end{lemma}
The proof is straightforward and only consists
in checking on the one hand, that
    \eqref{eq:HJBb} and \eqref{eq:FPKb}
    are respectively equivalent to
    \eqref{eq:HJBrd} and \eqref{eq:FPKrd}
    with $H^b$ instead of $H$;
    on the other hand, that
    \eqref{eq:muaa} and \eqref{eq:mub}
    are equivalent to \eqref{eq:murd} with $H^b$,
    where we take $\mu_b=\mu$
    and $\mu_{\aa}$ defined
    by \eqref{eq:muaa}.

The following existence theorem
is a direct consequence of
Lemma \ref{lem:equisol},
and Theorem \ref{thm:Exird}.
\begin{corollary}
    \label{thm:ExiMonob}
    If $L^{b}$ satisfies 
    \ref{hypo:Lreg}-\ref{hypo:Monogregx},
    %\ref{hypo:Lstrconv}-\ref{hypo:Llipmu},
    and $b$ satisfies
    \ref{hypo:binvert},
    there exists a solution to \eqref{eq:MFGCb}.
\end{corollary}
Theorem \ref{thm:bpower}, i.e. the existence
part of step \ref{step:MFGCb},
is a consequence of the latter existence result
in which the assumptions on $L^b$ are stated on $L$
instead, which makes them more tractable.
However, we have to
%replace \ref{hypo:Lstrconv}
%by \ref{hypo:Lbstrconv} and
make the additional
\ref{hypo:bbound}.

If $L$ and $b$ satisfy the assumptions of
Theorem \ref{thm:bpower},
it is straightforward to check that $L^b$
satisfies
\ref{hypo:Lreg}-\ref{hypo:Lbound}.
Therefore, Theorem \ref{thm:bpower}
is a consequence of Corollary \ref{thm:ExiMonob}.
%\ref{hypo:Lstrconv}-\ref{hypo:Lbound},
Finally, Theorem
\ref{thm:UniMonob}
and the uniqueness part of step
\ref{step:MFGCb}
are direct consequences of
Theorem \ref{thm:UniMono}
and Lemma \ref{lem:equisol}.

\begin{acknowledgement}
I wish to express my gratitude to Y. Achdou
and P. Cardaliaguet for technical advices,
insightful comments and corrections.
This research was partially supported by the ANR
(Agence Nationale de la Recherche) through
MFG project ANR-16-CE40-0015-01.
\end{acknowledgement}

\bibliographystyle{plain}
\bibliography{MFGbiblio}

\end{document}